\documentclass[10pt]{article}
\usepackage{amsmath}
\usepackage{amsthm}
\usepackage{amsfonts}
\usepackage{amssymb}
\usepackage{amscd}
\usepackage[all]{xy}
\usepackage[percent]{overpic}
\usepackage{graphicx}
\usepackage{xcolor}
\usepackage{enumerate}

\usepackage{hyperref}
\hypersetup{
    colorlinks=true,       
    linkcolor=blue,          
    citecolor=blue,        
    filecolor=blue,      
    urlcolor=blue           
}

    \newtheorem{Lem}{Lemma}[section]
    \newtheorem{Lem-Def}{Lemma-Definition}[section]
    \newtheorem{Prop}[Lem]{Proposition}
       
      \newtheorem*{thm}{Theorem}

    \newtheorem{Thm}[Lem]{Theorem}  
    \newtheorem{Cor}[Lem]{Corollary}

\theoremstyle{definition}
\font\smallsc=cmcsc10
\font\smallsl=cmsl10

    \newtheorem{Rem}[Lem]{Remark}

\newcommand{\Spec}{\text{Spec}\,}

\newcommand{\A}{\mathcal A}

\newcommand{\Z}{\mathcal Z}
\newcommand{\I}{\mathcal I}
\newcommand{\M}{\mathcal M}
\newcommand{\N}{\mathcal N}

\renewcommand{\L}{\mathcal L}
\renewcommand{\O}{\mathcal O}
\newcommand{\C}{\mathcal C}

\newcommand{\col}{\colon}

\newcommand{\ra}{\rightarrow}
\newcommand{\ol}{\overline}

\newcommand{\ul}{\underline}

\newcommand{\wt}{\widetilde}

\newcommand{\wh}{\widehat}
\newcommand{\dra}{\dashrightarrow}

\newcommand{\J}{\ol J}

\newcommand{\lra}{\longrightarrow}
\renewcommand{\:}{\colon}

\renewcommand{\l}{\ell}
\newcommand{\Jb}{\overline{\mathcal{J}_{\ul{e}}}}

\begin{document}

\title{On the geometry of Abel maps for nodal curves}
\author{Alex Abreu, Juliana Coelho and Marco Pacini\footnote{The third author was partially supported by CNPq, processo 300714/2010-6.}}
\date{}

\maketitle

\begin{abstract}
\noindent
In this paper we give local conditions to the existence of Abel maps for smoothings of nodal curves extending the Abel maps for the generic fiber. We use this result to construct Abel maps of any degree for nodal curves with two components.
\end{abstract}


\section{Introduction}

\subsection{History}
  Let $C$ be a smooth projective curve over an algebraically closed field and fix a point $P$ in $C$. A degree-$d$ Abel map is a map $\alpha^d_L\col C^d\to J_C$ from the product of $d$ copies of $C$ to its Jacobian $J_C$, sending $(Q_1,\ldots,Q_d)$ to the invertible sheaf $L(dP-Q_1-\ldots-Q_d)$, where $L$ is an invertible sheaf on $C$. It is classically known that this map encodes many geometric properties of the curve $C$. For instance, the Abel theorem states that the fibers of $\alpha^d_L$ are complete linear series on $C$, up to the action of the $d$-th symetric group. Thus, all possible embeddings of $C$ in projective spaces are known once we know its Abel maps. \par
   Often, to study linear series on smooth curves, we resort to degenerations to singular curves. Then, it is important to understand how linear series behave under such degenerations. It was through the study of these degenerations that Griffiths and Harris proved the celebrated Brill-Noether theorem in \cite{GH}, and later Gieseker proved Petri's conjecture in \cite{Gi}. This inspired the seminal work of Eisenbud and Harris \cite{EH} where they introduced the theory of limit linear series for curves of compact type. Nevertheless, a satisfatory general theory of limit linear series has not yet been obtained, although there are several works in this direction for curves with two components, for instance Coppens and Gatto in \cite{CG} and Esteves and Medeiros in \cite{EM}. More recently, Osserman gave in \cite{O} a more refined notion of limit linear series for a curve of compact type with two components.\par
   Since there is a relationship between linear series and Abel maps for smooth curves, it is expected an interplay between limit linear series and Abel maps for singular curves. This interplay was explored by Esteves and Osserman \cite{EO} for curves of compact type with two components, for which natural Abel maps exist. However Abel maps for singular curves have been constructed only in few cases: for irreducible curves in \cite{AK}, in degree one in \cite{CE} and \cite{CCE}, in degree two in \cite{Co}, \cite{CEP} and \cite{Pac}, and for curves of compact type and in any degree in \cite{CP}. \par
   There are two main compactifications of the Jacobian employed as targets of these Abel maps, namely Caporaso-Pandharipande's compactified Jacobian constructed in \cite{C} and \cite{Pand} and Esteves' compactified Jacobian constructed in \cite{E01} based on the previous work \cite{AK} of Altman and Kleiman. The principal goal of this paper is to construct Abel maps of any degree for curves with two components with Esteves' compactified Jacobian as target.\par
    In a different paper, we plan to describe the fibers of this map and unreveal their relationship with degenerations of linear series on smooth curves, in the spirit of the paper \cite{EO} and possibly compare the results with the work of Esteves and Medeiros in \cite{EM}.

\subsection{Main results}
  Let us explain in details our main results. Let $C$ be a nodal curve over an algebraically closed field $K$. Let $\pi\col\C\to B$ be a family of curves over $B:=\Spec(K[[t]])$ with smooth total space $\C$ and $C$ as special fiber. Let $\sigma\col B\to\C$ be a section of $\pi$ through its smooth locus and $\L$ be an invertible sheaf on $\C$ of relative degree $e$. Since the generic fiber of $\pi$ is a pointed smooth curve, there exists a rational Abel map $\alpha^d_\L\col \C^d\dashrightarrow \Jb$ from the product of $d$ copies of $\C$ over $B$ to the compactified Jacobian of $\pi$. Here, $\Jb$ is the fine moduli scheme, introduced by Esteves in \cite{E01}, parametrizing rank-$1$ torsion-free sheaves of degree $e$ that are $\sigma$-quasistable with respect to a polarization $\ul{e}$ of degree $e$ (see Section \ref{sec:Jac} for more details).\par
   We will resolve the map $\alpha^d_\L$ in the case where $C$ has two smooth components $C_1$ and $C_2$. To do that we construct a desingularization $\wt\C^d$ of $\C^d$ recursively on $d$. More precisely, we will perform a sequence of blowups along Weil divisors as follows.
   Set $\wt\C^1:=\C^1$. Assume that $\wt\C^d$ is constructed and let $\wt{\C}^{d+1}\to\wt{\C}^d\times_B\C$ be the sequence of blowups along the strict transforms of the following Weil divisors in the stated order
   \begin{equation*}
      \Delta_{d,d+1},\Delta_{d-1,d+1},\ldots,\Delta_{1,d+1},
   \end{equation*}
   and then
   \begin{equation*}
      C_1^{d+1}, C_1^{d}\times C_2,C_1^{d-1}\times C_2\times C_1, C_1^{d-1}\times C_2^2, \ldots, C_2^{d-1}\times C_1\times C_2,C_2^{d}\times C_1,C_2^{d+1},
      \end{equation*}
      where $\Delta_{i,d+1}$ is the ``$i$-th diagonal'', i.e., the image of the section $\wt{\C}^d\to\wt{\C}^d\times_B\C$, induced by the composition $\delta_i\col\wt{\C}^d\to\C^d\to \C$ of the  desingularization map with the projection onto the $i$-th factor.\par

\begin{thm}
 There exists a modular map $\ol{\alpha}^d_\L\colon\wt{\C}^d\to \Jb$ extending the map $\alpha^d_\L$.
\end{thm}
  We note that the order in which these Weil divisors are blown up is important to the resolution of the map. Indeed, it is not difficult to find examples in which a different sequence of blowups does not give rise to a resolution. Moreover, the desingularization $\wt\C^d$ is independent of the polarization $\ul{e}$ and the sheaf $\L$. We refer to \cite[Section 7]{CEP} for examples of resolutions for more general curves in degree 2.\par\smallskip

   In order to prove the result we consider a local formulation of our problem. Indeed, we note that the completion of the local ring of $\wt\C^d$ at a point is given by $K[[u_1,\ldots,u_{d+1}]]$, and, in the relevant cases, the map $\wt{\C}^d\to B$ is given by $t=u_1\cdot\ldots\cdot u_{d+1}$. For this reason we consider $S:=\Spec(K[[u_1,\ldots,u_{d+1}]])$ and the map $S\to B$ given by $t=u_1\cdot\ldots\cdot u_{d+1}$. Let $\C_S:=\C\times_B S$ and $\delta_1,\ldots,\delta_m$ be sections of $\pi_S\col\C_S\to S$. Since the generic fiber $\C_\eta$ over the generic point $\eta$ of $B$ is smooth, we have a rational map $\alpha_\L\col S\dashrightarrow \Jb$ sending $\eta_S$, the generic point of $S$, to the invertible sheaf 
 \[
 \L|_{\C_\eta}(m\sigma(\eta)-\delta_1(\eta_S)-\ldots-\delta_m(\eta_S)),
 \]
 where the sections $\delta_i$ are identified with their composition with the projection $\C_S\to \C$.

In Theorem~\ref{thm:map} we give numerical conditions to the existence of a map $\ol\alpha_\L\col S\to \Jb$ extending $\alpha_\L$. In fact, this result holds for curves with any number of components. \par
    
  To check that these conditions hold for a desingularization of $\C^d$, we need to understand its local geometry, that is, to understand how $\C^d$ behaves under the sequence of blowups performed. To do that, in Section~\ref{sec:des}, we give a local description of blowups along certain Weil divisors. Since this approach is local, it can be applied to curves with any number of components.\par
  Altough we only obtained a sequence of blowups for curves with two components, our techniques might be applied more generally to determine algorithmically whether or not a given sequence of blowups resolves the map $\alpha^d_\L$ for any nodal curve. This approach is similar to the one in \cite{CEP} where a script to determine the existence of the degree-$2$ Abel map was produced.

\subsection{Notation and terminology}
Throughout the paper we will use the following notations.\par

We work over an algebraically closed field $K$. A \emph{curve} is a connected, projective and reduced scheme of dimension 1 over $K$. We will always consider curves with  nodal singularities. A \emph{pointed curve} is a curve $C$ with a marked point $P$ in the smooth locus of $C$, usually denoted by $(C,P)$.

Let $C$ be a curve. We denote the irreducible components of $C$ by $C_1,\ldots, C_p$ and by $C^{sing}$ the set of its nodes. A \emph{subcurve} of $C$ is a union of irreducible components of $C$. If $Y$ is a proper subcurve of $C$, we let $Y^c:=\overline{C\setminus Y}$ and call it the \emph{complement} of $Y$. We denote $\Sigma_Y:=Y\cap Y^c$ and $k_Y:=\#\Sigma_Y$; a node in $\Sigma_Y$ is called an \emph{extremal} node of $Y$. A node $N$ of $C$ is \emph{external} if $N\in\Sigma_Y$ for some subcurve $Y$, otherwise the node is called \emph{internal}. We will always consider curves without internal nodes.

Given a map of curves $\phi\col C'\to C$ we say that an irreducible component of $C'$ is \emph{$\phi$-exceptional} if it is a smooth rational curve and is contracted by the map. A \emph{chain of rational curves of lenght $d$} is a curve which is the union of smooth rational curves $E_1,\ldots, E_d$ such that $E_i\cap E_j$ is empty if $|i-j|>1$ and $\#(E_i\cap E_{i+1})=1$. A \emph{chain of $\phi$-exceptional components} is a chain of $\phi$-exceptional curves.
 We define the curve $C(d)$ as the curve endowed with a map $\phi\col C(d)\to C$ such that $\phi$ is an isomorphism over the smooth locus of $C$, and the preimage of each node of $C$ consists of a chain of $\phi$-exceptional components of length $d$. If $(C,P)$ is a pointed curve we abuse notation denoting by $P$ its preimage in $C(d)$ so that $(C(d),P)$ is also a pointed curve.

A \emph{family of curves} is a proper and flat morphism $\pi\col\mathcal C\ra B$ whose fibers are curves. If $b\in B$, we denote $\C_b:=\pi^{-1}(b)$ its fiber. The family $\pi\col\C\to B$ is called \emph{local} if $B=\Spec(K[[t]])$,  \emph{regular} if $\C$ is regular and \emph{pointed} if it is endowed with a section $\sigma\col B\to \C$ through the smooth locus of $\pi$. A \emph{smoothing} of a curve $C$ is a regular local family $\pi\col\C\to B$ with special fiber $C$. Given a pointed smoothing $\pi\col\C\to B$ of a curve $C$ with section $\sigma\col B\to\C$, we define $P:=\sigma(0)$. If $f\:\C\ra B$ is a family of curves, we denote by $\C^d$ the product of $d$ copies of $\C$ over $B$.\par

Let $I$ be a coherent sheaf on a curve $C$. We say that $I$ is \emph{torsion-free} if its associated points are generic points of $C$. We say that $I$ is of \emph{rank 1} if $I$ is invertible on a dense open subset of $C$.  Each invertible sheaf on $C$ is a rank-$1$ torsion-free sheaf.  If $I$ is a rank-$1$ torsion-free sheaf, we call $\deg(I) := \chi(I )-\chi(\O_C)$ the \emph{degree}  of $I$. An invertible sheaf $I$ over $\phi\col C(d)\to C$ is \emph{$\phi$-admissible} if $\deg(I|_E)\in\{-1,0,1\}$ for every chain of $\phi$-exceptional components $E$.

We fix $B:=\Spec(K[[t]])$, $S:=\Spec(K[[u_1,\ldots, u_{d+1}]])$ and the map $S\to B$ given by $t=u_1\cdot u_2\cdot\ldots\cdot u_{d+1}$. We will call the closed point of both $B$ and $S$ by $0$, when no confusion may arise. Moreover, we denote by $Q_i$ the generic point of $V(u_i)$ in $S$. Given a smoothing $\pi:\C\to B$ of a curve $C$, define $\C_S:=\C\times_B S$ and let $\pi_S\col\C_S\to S$ be the induced map.

\section{Jacobians and Abel maps}
\label{sec:Jac}
  Let $\pi\col\C\to B$ be a pointed regular local family of nodal curves with section $\sigma\col B\to\C$. The \emph{degree-$e$ Jacobian} of $\pi$ is the scheme parametrizing the equivalence classes of degree-$e$ invertible sheaves on the fibers. In general, this scheme is neither proper nor of finite type. To solve these issues we resort to rank-$1$ torsion-free sheaves and to stability conditions.\par
  Let $C$ be a nodal curve with $p$ irreducible components $C_1,\ldots,C_p$ and $P$ be a smooth point of $C$. A \emph{polarization of degree $e$ on $C$} is any $p$-tuple of rational numbers $\ul{e}=(e_1,\ldots,e_p)$ summing up to $e$.  Let $Y$ be a proper subcurve of $C$. We set 
  \[
  e_Y:=\sum_{C_i\subset Y} e_i.
  \]
   Let $I$ be a rank-$1$ degree-$e$ torsion-free sheaf on $C$. We define the sheaf $I_Y$ as the sheaf $I|_Y$ modulo torsion. We say that $I$ is \emph{$P$-quasistable over $Y$ (with respect to $\ul{e}$)} if the following condition holds
\begin{eqnarray*}
\frac{-k_Y}{2}< \deg(I_Y)-e_Y \leq \frac{k_Y}{2}, & \text{if}& P\in Y,\\
\frac{-k_Y}{2}\leq \deg(I_Y)-e_Y < \frac{k_Y}{2}, & \text{if}& P\notin Y.
\end{eqnarray*}
Equivalently, $I$ is $P$-quasistable over $Y$ if the following conditions hold
\begin{eqnarray*}
\frac{-k_Y}{2}< \deg(I_Y)-e_Y \quad\text{and}\quad \frac{-k_Y}{2}\leq \deg(I_{Y^c})-e_{Y^c}& \text{if}& P\in Y, \\
\frac{-k_Y}{2}\leq \deg(I_Y)-e_Y \quad\text{and}\quad \frac{-k_Y}{2}< \deg(I_{Y^c})-e_{Y^c}& \text{if}& P\notin Y.
\end{eqnarray*}
Note that $I$ is $P$-quasistable over $Y$ if and only if it is over $Y^c$. \par
 
 We say that $I$ is \emph{$P$-quasistable} over $C$ if it is $P$-quasistable over every proper subcurve of $C$. Since the conditions are additive on connected components it is enough to check them over connected subcurves. In fact, it is easy to see that it suffices to check on connected subcurves with connected complement.\par
 Given the map of curves $\phi\col C(d)\to C$ and a polarization $\ul{e}$ over $C$, we define the polarization $\ul{e}(d)$ over $C(d)$ simply by $e(d)_Y=e_{\phi(Y)}$ if $\phi(Y)$ is not a point and $e(d)_Y=0$ otherwise, where $Y$ is a irreducible component of $C(d)$. From now on fix a polarization $\ul{e}$ of degree $e$ on $C$, and its induced polarizations $\ul{e}(d)$.\par
  Let $\pi\col\C\to B$ be a pointed regular local family of nodal curves with section $\sigma\col B\to \C$. We say that a sheaf $\I$ over $\C$ is \emph{$\sigma$-quasistable} if it restricts to a torsion-free rank-$1$ sheaf over each fiber of $\pi$ and if its restriction to the special fiber $C$ of $\pi$ is $\sigma(0)$-quasistable. The \emph{degree-$e$ compactified Jacobian of $\pi$} is the scheme $\Jb$ parametrizing $\sigma$-quasistable sheaves over $\C$ of degree $e$. This scheme is proper and of finite type (see \cite[Thms A and B]{E01}) and it represents the contravariant functor $\mathbf{J}$ from the category of locally Noetherian $B$-schemes to sets, defined on a $B$-scheme $S$ by
\[
\mathbf{J}(S):=\{\sigma_S\text{-quasistable sheaves of degree $e$ over } \C\times_B S\stackrel{\pi_S}\lra S\}/\sim
\]
where $\sigma_S$ is the pullback of the section $\sigma$ and $\sim$ is the equivalence relation given by $I_1\sim I_2$ if and only if there exists an invertible sheaf $M$ on $S$ such that $I_1\cong I_2\otimes \pi_S^*M$. \par

\begin{Prop}
\label{prop:pushforward}
Let $(C,P)$ be a pointed nodal curve and consider $\phi\col C(d)\to C$. Let $L$ be a line bundle over $C(d)$ that is $\phi$-admissible and $P$-quasistable over each subcurve $Y$ of $C(d)$ such that  $Y$ and $Y^c$ are connected and neither is contracted by $\phi$. Then the sheaf $\phi_*(L)$ is $P$-quasistable.
\end{Prop}
\begin{proof}
Fix $\C\to B$ a smoothing of $(C(d),P)$, let $\L$ be a line bundle on $\C$ such that $\L|_{C(d)}=L$. By \cite[Propositions 5.2 and 5.3]{CEP} it suffices to show that exists a twister $\O_{\C}(Z)$, with $Z$ is a divisor supported on the exceptional components of $\phi$, such that $(\L\otimes \O_{\C}(Z))|_{C(d)}$ is $P$-quasistable.\par
    The divisor $Z$ is effective and can be algorithmically computed as follows. Recall that $\L$ is admissible if and only if its degree on each chain of $\phi$-exceptional components is $-1$, $0$ or $1$. We define invertible sheafs $\L_i$ inductively. Set $\L_0:=\L$. For a maximal chain of $\phi$-exceptional components $E$ over some node of $C$, let $W_{E,i}$ be the (possibly empty) maximal subchain of $E$ such that $\deg(\L_{i-1}|_{W_{E,i}})=1$. Define 
   \[
   Z_i:=\bigcup_{E} W_{E,i}
   \]
and $\L_i:=\L_{i-1}(Z_i)$.\par

  We claim that $\L_i$ is admissible and that $Z_{i+1}$ is empty or strictly contained in $Z_i$. Indeed, fix a maximal chain $E=E_1\cup\ldots\cup E_d$ and let $W_{E,i}=E_\ell\cup\ldots\cup E_h$. We have that $\deg(\L_{i-1}|_{W_{E,i}})=1$ and the maximality of $W_{E,i}$ implies that either $\ell=1$ or
\[
\deg(\L_{i-1}|_{E_{\ell-1}})=-1\quad\text{and} \quad\deg(\L_{i-1}|_{E_k})=0,\quad\text{for every}\quad1\leq k\leq \ell-2;
\]
also either $h=d$ or 
\[
\deg(\L_{i-1}|_{E_{h+1}})=-1\quad\text{and}\quad\deg(\L_{i-1}|_{E_k})=0,\quad\text{for every}\quad h+2\leq k\leq d.
\]
 This implies that either $\ell=1$ or 
\[
\deg(\L_{i}|_{E_k})=0,\quad\text{for every}\quad1\leq k\leq \ell-1;
\]
also either $h=d$ or
\[
\deg(\L_{i}|_{E_k})=0,\quad\text{for every}\quad h+1\leq k\leq d.
\]
Therefore, we see that $\L_i$ is admissible.\par

Moreover we have
\[
\deg(\L_{i}|_{W_{E,i}})=-1,
\]
meaning that there exists $\ell'$ and $h'$ such that 
\[
\deg(\L_{i}|_{E_{\ell'}})=-1\quad\text{and} \quad\deg(\L_{i}|_{E_k})=0,\quad\text{for every}\quad \ell\leq k\leq \ell'-1
\]
and
\[
\deg(\L_{i}|_{E_{h'}})=-1\quad\text{and} \quad\deg(\L_{i}|_{E_k})=0,\quad\text{for every}\quad h'+1\leq k\leq h.
\]
Clearly $W_{E,i+1}=E_{\ell'+1}\cup\ldots\cup E_{h'-1}$ (which may be empty if $\ell'=h'$), and then $W_{E,i+1}$ is strictly contained in $W_{E,i}$. This concludes the proof of the claim.\par
  Define 
\[
Z:=\sum_{i\geq1} Z_i,
\]
Now it is enough to prove that $\N:=\L\otimes \O_\C(Z)$ restricted to $C(d)$ is $P$-quasistable. Let $Y$ be a connected subcurve of $C(d)$ with connected complement. If $Y$ is contracted by the map $\phi$, then $Y$ is a chain of exceptional components, hence, since $\N$ is admissible and there is no chain of exceptional components over which $\N$ has degree $1$, it follows that $\deg(\N|_Y)\in\{-1,0\}$. This proves that $\N$ is $P$-quasistable over $Y$ and $Y^c$.\par
 Now assume that neither $Y$ nor $Y^c$ is contracted by $\phi$. For every $N\in\Sigma_{\phi(Y)}$, we define
 \[
 Y^\circ:=\ol{Y\setminus \bigcup_{N\in \Sigma_{\phi(Y)}}\phi^{-1}(N)},\quad N^\circ:=\phi^{-1}(N)\cap Y^\circ,
 \]
 \[
  E_N:=\ol{(\phi^{-1}(N)\setminus\{N^\circ\})\cap Y}\quad\text{and}\quad E_Y:=\bigcup_{N\in\Sigma_{\phi(Y)}} E_N. 
 \]
Note that $Y^\circ=\ol{Y\setminus E_Y}$, and hence
\[
\deg(\N|_Y)=\deg(\N|_{Y^\circ})+\deg(\N|_{E_Y}).
\]
Moreover, we have
\[
\deg(\N|_{Y^\circ})=\deg(\L|_{Y^\circ})+\sum_{N\in\Sigma_{\phi(Y)}}\epsilon_N,
\]
where $\epsilon_N$ is $1$ if $N^\circ\in Z$ and $0$ otherwise. Note that if $\epsilon_N=0$ then either there exists a chain of exceptional components $E_N'$ such that $\deg(\L|_{E_N'})=-1$ and $E_N'\cap Y^\circ \neq\emptyset$ or the degree of $\L$ over every chain of exceptions components contained in $\phi^{-1}(N)$ is zero, and in this case define $E_N'=\emptyset$. Define 
\[
Y':=Y^\circ\cup\underset{\epsilon_N=0}{\bigcup_{N\in \Sigma_{\phi(Y)}}} E_N', 
\]
then
\begin{eqnarray*}
\deg(\N|_{Y})&=&\deg(\L|_{Y^\circ})+\sum_{N\in\Sigma_{\phi(Y)}}\epsilon_N+\deg(\N|_{E_Y})\\
              &=&\deg(\L|_{Y^\circ})+\sum_{\epsilon_N=0}\deg(\N|_{E_N})+ \sum_{\epsilon_N=1}(\epsilon_N+\deg(\N|_{E_N}))\\
             &\geq&\deg(\L|_{Y'}),
\end{eqnarray*}
implying that
\[
\deg(\N|_Y)-e(d)_Y \geq \deg(\L|_{Y'})-e(d)_{Y'}.
\]
We can repeat the same process for $Y^c$, obtaining a subcurve ${Y^c}'$ satisfying
\[
\deg(\N|_{Y^c})-e(d)_{Y^c}\geq\deg(\L|_{{Y^c}'})-e(d)_{{Y^c}'}.
\]\par
 Since both $Y'$ and ${Y^c}'$ are connected with connected complement and are not contracted by $\phi$, it follows that $\L$ is $P$-quasistable over $Y'$ and ${Y^c}'$, and therefore $\N$ is $P$-quasistable over $Y$. The proof is complete.\end{proof}
   
   Let $\pi\col\C\to B$ be a pointed regular local family of nodal curves with section $\sigma\col B\to\C$. Let $C$ be the special fiber of $\pi$ with irreducible components $C_1,\ldots, C_p$. We define $\dot\C$ as the smooth locus of $\pi$ and $\dot C_i:=C_i\cap \dot\C$. Set $\dot\C^d:=\dot\C\times_B\dot\C\times_B\ldots\times_B\dot\C$, the product of $d$ copies of $\dot\C$ over $B$. Note that the special fiber of $\dot\C^d\to B$ is
\[
\coprod_{1\leq i_1,\ldots,i_d\leq p} \dot C_{i_1}\times\ldots\times \dot C_{i_d}.
\]
For each $d$-tuple $\underline{i}=(i_1,\ldots,i_d)$ define $\dot C_{\underline{i}}:=\dot C_{i_1}\times\ldots\times \dot C_{i_d}$. 
Let $\L$ be a degree-$e$ invertible sheaf over $\C$. There exists a degree-$d$ Abel map from $\dot\C^d$ to the degree-$e$ Jacobian of $\pi$ simply sending the $d$-tuple $(Q_1,\ldots,Q_d)$ over $b$ to the invertible sheaf 
\begin{equation}
\label{eq:l}
\L|_{\C_b}(d\cdot\sigma(b)-Q_1-\ldots-Q_d).
\end{equation}
We want to extend this Abel map to $\C^d$, and it is convenient to consider the degree-$e$ compactified Jacobian $\Jb$ as target. However, the sheaf \eqref{eq:l} may not be $\sigma(b)$-quasistable and thus we do not even have a map from $\dot\C^d$ to $\Jb$ defined as above. To solve this we use twisters and the fact that $\Jb$ represents the functor $\mathbf{J}$.\par

Indeed, form the fiber diagram
  \[
\begin{CD}
\dot\C^d\times_B\C @>f>> \C \\
@V\pi_dVV   @VV\pi V\\
\dot\C^d @>>> B
\end{CD}
\]\smallskip
By \cite[Thm 32, (4)]{E01}, for each $\underline{i}$ there exists a divisor
\begin{equation}
\label{eq:z}
Z_{\underline{i}}=\sum_{j=1}^p\ell_{\underline{i},j}\cdot \dot C_{\underline{i}}\times C_j
\end{equation}
of $\dot\C^d\times_B\C$ such that the invertible sheaf $\M$ defined as
\[
\M:=f^*\L\otimes\O_{\dot\C^d\times_B\C}\left(d\cdot f^*\sigma(B)- \sum_{i=1}^d\Delta_{i,d+1}\right) \otimes\O_{\dot\C^d\times_B\C} \left(-\sum_{\underline{i}}Z_{\underline{i}}\right)
\]
is $f^*\sigma$-quasistable, where $\Delta_{i,d+1}$ is the preimage of the diagonal via the projection map $\dot\C^d\times_B\C\to \C\times_B\C$ onto the $i$-th and $d+1$-th factor. This $f^*\sigma$-quasistable sheaf $\M$ induces the Abel map
\[
\alpha_\L^d\col\dot\C^d\lra\Jb.
\]
 In this paper we give conditions to determine when this map extends to a suitable desingularization of $\C^d$.

\section{Desingularizations}
\label{sec:des}

 Given a smoothing $\pi\col\C\to B$ of a curve $C$ and $N$ a node of $C$, we can write the completion of the local ring of $\C$ at $N$ as
\[
\wh{\O}_{\C,N}\simeq K[[x,y]].
\]
The map $\pi\col\C\to B$ is, locally around $N$, given by $xy=t$. In this section we study the geometry of this local map and its formation with base change. In Figure 1 we collect all the relevant results in an explicit example.\par
 Recall that we defined $S=\Spec(K[[u_1,\ldots,u_{d+1}]])$ and a map $S\to B$ given by $t=u_1\cdot \ldots\cdot u_{d+1}$. Define $T:=\Spec(K[[x,y]])$ and the map $T\to B$ given by $t=xy$. Let $T_S:=T\times_B S$. Clearly, we have
 \[
 T_S=\Spec\left(\frac{K[[u_1,\ldots, u_{d+1},x,y]]}{(xy-u_1\cdots u_{d+1})}\right).
 \]
Given a subset $A$ of $\{1,\ldots,d+1\}$, we define
\[
u_A:=\prod_{j\in A} u_j.
\]

To desingularize $T_S$, we will blowup Weil divisors of type $D_A:=V(x,u_A)$, where $A$ is a proper nonempty subset of $\{1,\ldots,d+1\}$. More precisely, given a collection of proper nonempty subsets  $\A:=(A_1,\ldots,A_k)$ of $\{1,\ldots,d+1\}$,  we will perform a sequence of blowups 
\begin{equation}
\label{eq:blowup}
\phi\col\wt{T}_{S}^{\A}:=\wt{T}_{S}^{k}\stackrel{\phi_k}{\lra} \wt{T}_{S}^{k-1}\stackrel{\phi_{k-1}}{\lra}\cdots\stackrel{\phi_2}{\lra} \wt{T}_{S}^{1}\stackrel{\phi_1}{\lra} \wt{T}_{S}^{0}:=T_S,
\end{equation}
where the map $\phi_i$ is the blowup of the strict transform $\wt{D}_{A_i}$ of $D_{A_i}$ via the composition map $\phi_1\circ\ldots\circ\phi_{i-1}$.

\begin{Rem}
\label{rem:diagonal}
Note that the local equations of the blowup of $T_S$ along $D_A$ are given by $\alpha x-\alpha'u_A=0$ and $\alpha'y-\alpha u_{A^c}=0$ where $(\alpha:\alpha')$ are coordinates of $\mathbb{P}^1$. It is easy to see that if we blow up $V(y,u_{A^c})$ we will obtain the same equations. Therefore, blowing up $V(y,u_{A^c})$ is equivalent to blowing up $D_A$.\par
  The same property holds for the blowup along $V(x-u_{A^c},y-u_A)$. Indeed the local equation of such blowup is
  \begin{equation}\label{eq:uAuAc}
  \alpha(x-u_{A^c})-\alpha'(y-u_A)=0.
  \end{equation}
  Nevertheless, we know that the relation $xy=u_Au_{A^c}$ holds, and this relation is equivalent to $x(y-u_A)=u_A(u_{A^c}-x)$. Hence, we can simplify the equation \eqref{eq:uAuAc}  to the equations
  \[
  \alpha x+\alpha'u_A=0\quad\text{and}\quad \alpha y+\alpha' u_{A^c}=0,
 \]
 which, up to sign, are the same equations for the blowup along $D_A$. 
 This justifies why in the sequel we only consider blowups along divisors of type $D_A$.
\end{Rem}

Let $\A=(A_1,\ldots, A_k)$ be a collection of subsets of $\{1,\ldots, d+1\}$ and $A$ be a subset of $\{1,\ldots,d+1\}$. Assume that $\wt{T}_S^{\A}$ is obtained by a sequence of blowups of $T_S$ as in \eqref{eq:blowup}.  Also, let $S_A$ be the complement of $V(u_A)$ in $S$. We have $S_A=\Spec(K[[u_1,\ldots,u_{d+1}]]_{u_A})$. Define $T_{S_A}:=T\times_B {S_A}$ and $\wt{T}_{S_A}^\A:=\wt{T}_S^\A\times_S S_A$. We have the fiber diagram
\[
\begin{CD}
\wt{T}_{S_A}^{\A} @>>> \wt{T}_S^{\A} \\
@VVV   @VVV\\
T_{S_A} @>>> T_S\\
@VVV @VVV\\
S_A @>>> S
\end{CD}
\]

We call a collection $\A:=(A_1,\ldots,A_k)$ of subsets of a finite set $F$ a \emph{smooth collection for $F$} if for every distinct $i,j\in F$,  there exists $\l$ such that either $j\in A_\l$ and $i\notin A_\l$, or $i\in A_\l$ and $j\notin A_\l$.

\begin{Prop}
\label{prop:smooth}
The scheme $\wt{T}_{S}^{\A}$ is smooth if and only if $\A$ is a smooth collection for $\{1,\ldots,d+1\}$. Moreover, in that case, the inverse image of the closed point of $T_S$ in $\wt{T}_S^{\A}$ is a chain of $d$ rational curves.
\end{Prop}

\begin{proof} First assume that $\wt{T}_S^{\A}$ is smooth. Consider the open subscheme $S_{\{i,j\}^c}$ with $i,j=1,\ldots, d+1$ distinct. Note that $\wt{T}_{S_{\{i,j\}^c}}$ is smooth and $T_{S_{\{i,j\}^c}}$ is not. Hence there exists one of the divisors $D_{A_\ell}=V(x,u_{A_\ell})$ such that the restriction to $T_{S_{\{i,j\}^c}}$ is not Cartier. However the equation of $T_{S_{\{i,j\}^c}}$ is $xy-u_iu_j=0$, and hence the restriction of $D_{A_\ell}$ is not Cartier only if $i\in A_\ell$ and $j\notin A_\ell$ or $i\notin A_\ell$ and $j\in A_\ell$.\par 

Assume now that $\A$ is smooth. An open covering for $\wt{T}_{S}^{1}$ is given by
\begin{equation}
\label{eq:Ulocal}
U:=\Spec\left(\frac{K[[x,y,u_1,\ldots,u_{d+1},\alpha]]}{(\alpha x-u_{A_1},y-\alpha u_{A_1^c})}\right)=\Spec\left(\frac{K[[x,u_1,\ldots,u_{d+1},\alpha]]}{(\alpha x-u_{A_1})}\right)
\end{equation}
\begin{equation}
\label{eq:Vlocal}
V:=\Spec\left(\frac{K[[x,y,u_1,\ldots,u_{d+1},\alpha']]}{(x-\alpha'u_{A_1},\alpha'y- u_{A_1^c})}\right)=\Spec\left(\frac{K[[y,u_1,\ldots,u_{d+1},\alpha']]}{(\alpha'y-u_{A_1^c})}\right)
\end{equation}

We claim that the strict transform $\wt{D}_{A_2}$ in $\wt{T}_{S}^{1}$ is given locally, up to Cartier divisors, by 
\[
V\left(x,u_{A_1\cap A_2}\right)\subset U \quad\quad\text{and}\quad\quad V\left(\alpha',u_{A_1^c\cap A_2}\right)\subset V.
\]
To see this, just note that 
\[
(x,u_{A_2})=\bigcap_{j\in A_2} (x,u_j).
\]
Hence we need only analyze the strict transforms $\wt{V}(x,u_j)$ of $V(x,u_j)$. Since the strict transform is contained in the inverse image, we get
\[
(x,u_j)\subset I(\wt{V}(x,u_j)\cap U).
\]
(Here, note that we are using the same notation for the coordinates in both $U$ and $T_S$.)
Therefore, by Equation~\ref{eq:Ulocal}, we readily see that $(x,u_j)$ has codimension $1$ in $U$ if and only if $j\in A_1$, which implies that $\wt{V}(x,u_j)$ is empty if $j\in A_1^c$. 
Arguing similarly for $V$  we see that
\[
(\alpha'u_{A_1},u_j)\subset I(\wt{V}(x,u_j)\cap V).
\]
Therefore, if $j\in A_1$, using Equation~\ref{eq:Vlocal} we get that $I(\wt{V}(x,u_j))=(u_j)$ in $V$, and hence $\wt{V}(x,u_j)$ is Cartier in $V$. Otherwise if $j\in A_1^c$, then we have
\[
(\alpha'u_{A_1},u_j)=(\alpha',u_j) \cap\bigcap_{i\in A_1} (u_i,u_j).
\]
Since $(u_i,u_j)$ has codimension $2$ in $V$, we conclude that 
\[
I(\wt{V}(x,u_j)\cap V)=(\alpha',u_j).
\]\par

 To sum up: The strict transform $\wt{V}(x,u_j)$ of $V(x,u_j)$ has empty intersection with $U$ (respectively is a Cartier divisor in $V$) if $j\in A_1^c$ (respectively if $j\in A_1$). Otherwise if $j\in A_1$ (respectively if $j\in A_1^c$), then this intersection is given by $(x,u_j)$ in $U$ (respectively, by $(\alpha',u_j)$ in $V$). The proof of the claim is complete.\par

We proceed now by induction on $d$. First, we split Sequence~\eqref{eq:blowup} in two, using the open covering $\wt{T}_{S}^{1}=U\cup V$. Define 
\[
U_\ell:=(\phi_2\circ\cdots\circ\phi_\ell)^{-1}(U)\quad\text{and}\quad V_\ell:=(\phi_2\circ\cdots\circ\phi_\ell)^{-1}(V)
\]
and the sequence
\[
U_k\stackrel{\phi_k}{\lra} U_{k-1}\stackrel{\phi_{k-1}}{\lra}\cdots\stackrel{\phi_3}{\lra} U_{2}\stackrel{\phi_2}{\lra} U,
\]
where the map $\phi_\ell$ is the blowup along the intersection $\wt{D}_{A_\ell}\cap U_\ell$. By the claim above, the strict transforms of $D_{A_2},D_{A_3},\ldots,D_{A_k}$ via the map $U\to T_{S}$ are given by equations
\[
(x,u_{A_1\cap A_2}),(x,u_{A_1\cap A_3}),\ldots,(x,u_{A_1\cap A_k}).
\]
Now, we just observe that the collection 
\[
\A_U:=(A_1\cap A_2,A_1\cap A_3,\ldots, A_1\cap A_k)
\]
is a smooth collection for $A_1$. Since $|A_1|\leq d$, by induction hypothesis $U_k$ is smooth.\par

For $V$ the argument is similar, just note that Cartier divisors may appear as components of the strict transforms, but they do not give contributions to the blowups. This proves the smoothness.\par

To prove the second statement, we still proceed by induction on $d$. By induction hypothesis, the inverse image of the closed point in $U$ via the map $U_k\to U$ is a chain of $|A_1|-1$ rational curves, and the inverse image of the closed point in $V$ via the map $V_k\to V$ is a chain of $|A_1^c|-1$ rational curves. Since the blowup of $V(x,u_{A_1})$ adds exactly one rational curve, we get the result.
\end{proof}


\begin{Cor}
\label{cor:chain}
Let $\A$ be a smooth collection for $\{1,\ldots,d+1\}$. Then, the inverse image of the closed point in $T_{S_A}$ via the map $\wt{T}_{S_A}^\A\to T_{S_A}$ is a chain of $d-|A|$ rational curves.
\end{Cor}
\begin{proof}
Just note that $(A_1\cap A,\ldots,A_k\cap A)$ is a smooth collection for $A$.
\end{proof}

Let $\A=(A_1,\ldots, A_k)$ be a smooth collection for a finite set $F$. We define the \emph{$\A$-ordering} of $F$ as follows: \par
Let $m,n$ be distinct elements of $F$. We say  that $m<_\A n$ if there exists $j$ such that $m\in A_j$, $n\notin A_j$ and for every $i<j$ we have that either $\{m,n\}\subset A_i$ or $\{m,n\}\subset A_i^c$. Since $\A$ is smooth, the ordering $<_\A$ is a complete ordering of $F$.\par

 Fix a smooth collection $\A$ of $\{1,\ldots,d+1\}$. Let $\wt{T}_S^\A$ be the desingularization of $T_S$ obtained via $\A$. The inverse image of the closed point in $S$ via the map $T_S\to S$ is the germ of nodal curve given by two branches $x=0$ and $y=0$. It follows from Proposition~\ref{prop:smooth} that the inverse image of the closed point of $S$ via the map $\wt{T}_S^{\A}\to S$ is the union of these two branches, but with the singular point replaced with a chain of $d$ rational curves. We denote by $N_1,\ldots, N_{d+1}$ the nodes lying on the chain, where $N_1$ is the one in $y=0$ and $N_{d+1}$ is the one in $x=0$, and by $E_1,\ldots, E_d$ the rational curves, where $\{N_i,N_{i+1}\}\subset E_i$.\par

 From now on all the strict transforms will be via the map $\phi:\wt{T}_S^\mathcal{A}\to T_S$.
 
\begin{Lem}
\label{lem:empty}
Let $\A$ be a smooth collection for $\{1,\ldots, d+1\}$. If $i<_A j$ then the strict tranforms via $\wt{T}_S^\A\to T_S$ of $V(x,u_j)$ and $V(y,u_i)$ do not intersect. 
\end{Lem}

\begin{proof}
Keep the notation of the proof of Proposition~\ref{prop:smooth}. We proceed by induction on $d$. 
We analyze first the case where $i\in A_1$ and $j\notin A_1$. It follows from the proof of Proposition~\ref{prop:smooth} that the strict transform of $V(x,u_j)$ is empty in $U$. Similarly, the strict transform of $V(y,u_i)$ is empty in $V$. Therefore there is no intersection in this case.\par

On the other hand, if $i,j\in A_1$, then the equations of the strict transforms of $V(x,u_j)$ and $V(y,u_i)$ in $U$ become $(x,u_j)$ and $(\alpha,u_i)$. By induction hypothesis the intersection of these strict transforms in $U_k$ is empty. Since the strict transform of $V(y,u_i)$ is empty in $V$, we are done also in this case. The case $i,j\in A_1^c$ is similar.
\end{proof}

 The singular locus of the map $\wt{T}_S^\A\to S$ consists of $d+1$ connected components, each one of which dominates one region of type $V(u_j)\subset S$, for some $j=1,\ldots,d+1$. Indeed, if we keep the same notation of Proposition~\ref{prop:smooth}, then the singular locus lying over $V(u_j)$ of the map $T_S\to S$  is contained in both $V(x,u_j)$ and $V(y,u_j)$, and if $i\neq j$, then either the strict transforms of $V(x,u_j)$ and $V(y,u_i)$ do not intersect, or the strict transforms of $V(x,u_i)$ and $V(y,u_j)$ do not intersect.  We will denote by $\Sigma_j$  the connected component of the singular locus of the map $\wt{T}_S^\A\to S$ that dominates $V(u_j)$. 
 
 In the sequel we will often use the following fact: If $N_i\in\Sigma_j$ then the rational curves $E_i,\ldots, E_d$ are contained in the strict transform of $V(x,u_j)$.

\begin{Prop}
\label{prop:node}
Let $\A$ be a smooth collection for $\{1,\ldots,d+1\}$. If $\eta$ is a permutation of $1,\ldots,d+1$ such that $\eta(1)<_\A\eta(2)<_\A\ldots<_\A\eta(d+1)$, then $N_i\in \Sigma_{\eta(i)}$.
\end{Prop}
\begin{proof}
Since the $\Sigma_j$'s are disjoint and each node belongs to at least one of them, it is clear that each node $N_i$ is contained in exactly one $\Sigma_j$; we denote such index by $j:=\tau(i)$.
\begin{figure}[t]
\label{fig:desing}

\begin{minipage}{\linewidth}
\begin{minipage}{0.57\linewidth}
\begin{overpic}[width=\linewidth, trim=0 1cm 0 0, clip]{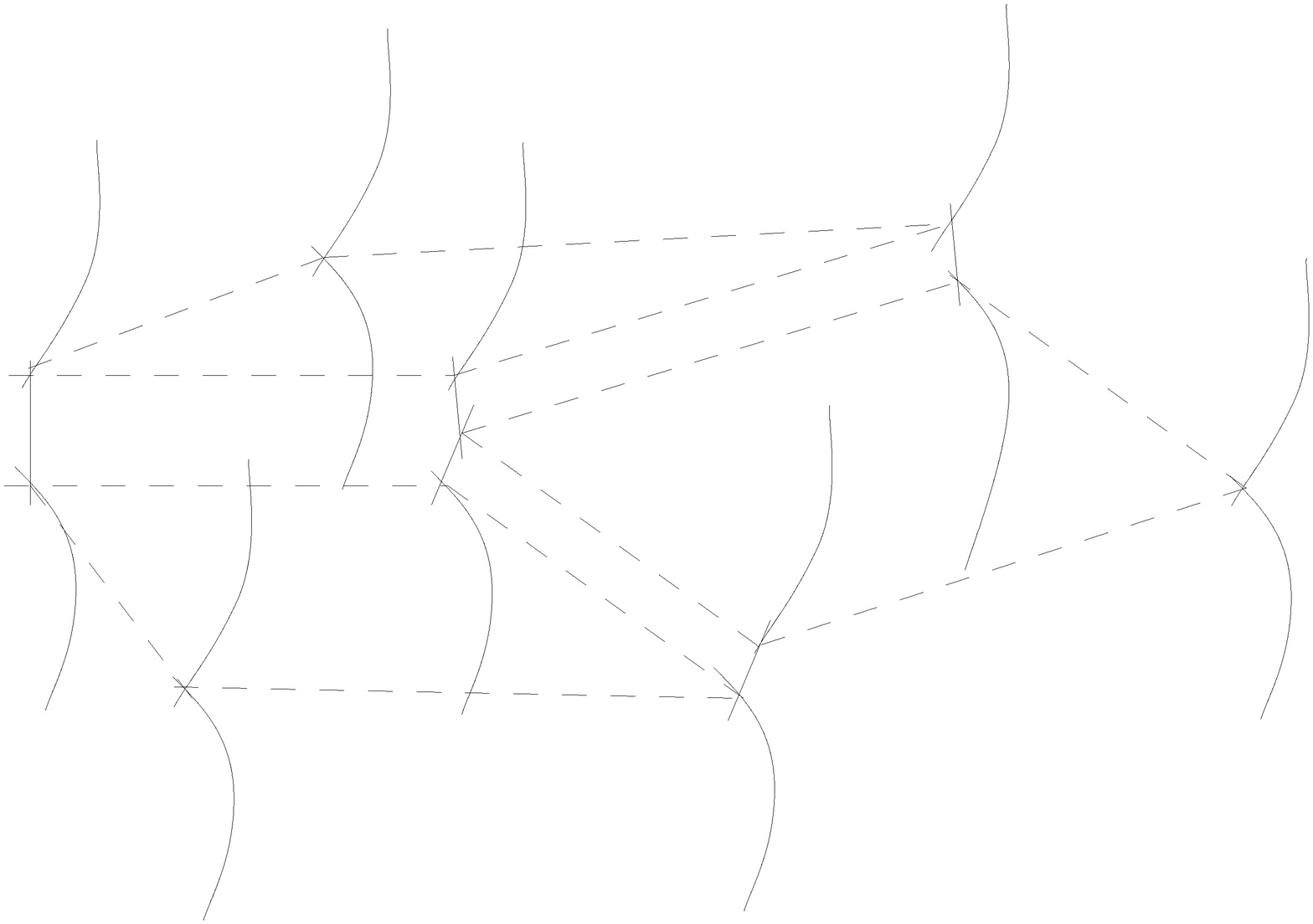}
\put(25,24){\small$\Sigma_3$}
\put(65,37){\small$\Sigma_2$}
\put(17,44){\small$\Sigma_1$}
\put(29,44){\tiny$N_1$}
\put(29,38){\tiny$N_2$}
\put(29,32){\tiny$N_3$}
\end{overpic}
\begin{overpic}[width=\linewidth, trim=0 1cm 0 0, clip]{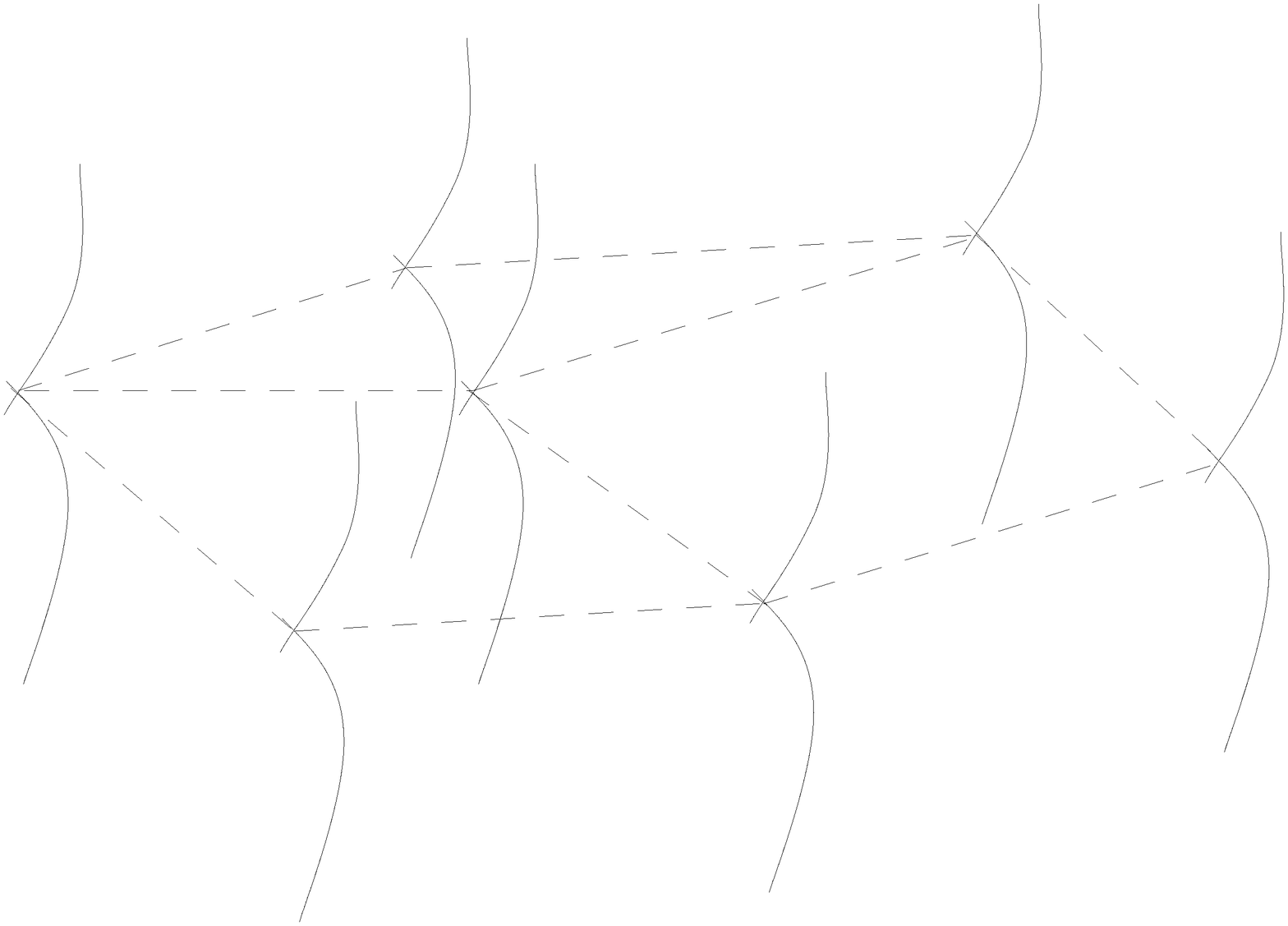}
\put(87,25){\tiny$x=0$}
\put(87,45){\tiny$y=0$}
\end{overpic}
\begin{overpic}[width=\linewidth, trim=0 7cm 0 5cm, clip]{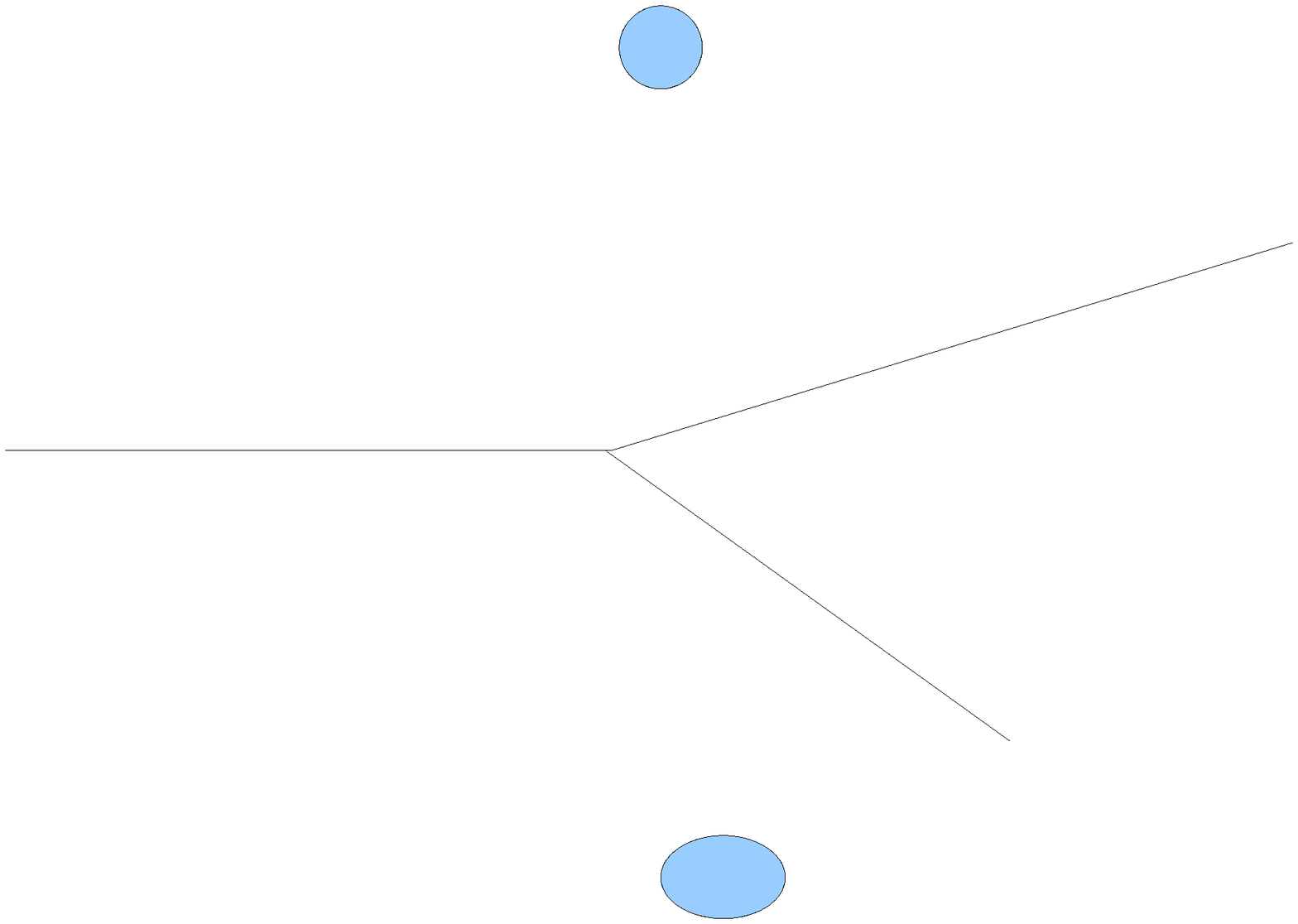}
\put(35,25){\small$V(u_1)$}
\put(70,15){\small$V(u_2)$}
\put(25,5){\small$V(u_3)$}
\end{overpic}

\end{minipage}
\begin{minipage}{0.42\linewidth} 
\small\textbf{Figure 1.} In this picture we describe the maps 
 \[
 \wt{\pi}_S\col\wt{T}_S^\A\stackrel{\phi}\longrightarrow T_S\stackrel{\pi_S}\lra S
 \]
  over the special point of $B$, in the case $d=2$ for $\A=(\{1\},\{2\})$.\par
  
\hspace{5pt} At the bottom we have depicted the variety $S$ with its divisors $V(u_i)$.\par
   
\hspace{5pt} At the middle, the variety $T_S$ with the branch $y=0$ being the top one. The inverse image of each $V(u_i)$ via the map $\pi_S$ is the union of the Weil divisors $V(x,u_i)$ and $V(y,u_i)$.  The map $\phi$ is the blowup of $T_S$ along $V(x,u_1)$ and then $V(x,u_2)$.\par
   
 \hspace{5pt}At the top we have the variety $\wt{T}_S^\A$. The dotted lines bound the singular loci $\Sigma_1$, $\Sigma_2$ and $\Sigma_3$ of the map $\wt{\pi}_S$. The permutation $\eta$ in this case is the identity, thus the node $N_i$ belongs to $\Sigma_i$, see Proposition~\ref{prop:node} and Corollary~\ref{cor:rational}. Note that in the central fiber we have two $\phi$-exceptional curves, see Proposition~\ref{prop:smooth} and Corollary~\ref{cor:chain}. The node $N_1$ belongs to the strict transforms of $V(x,u_1)$, $V(y,u_1)$, $V(y,u_2)$ and $V(y,u_3)$, while $N_2$ belongs to the ones of $V(x,u_1)$, $V(x,u_2)$, $V(y,u_2)$ and $V(y,u_3)$, and finally $N_3$ belongs to the ones of $V(x,u_1)$, $V(x,u_2)$, $V(x,u_3)$ and $V(y,u_3)$, see Corollary~\ref{cor:rational}.
\end{minipage}
\end{minipage}
\end{figure}

 Without loss of generality we may assume that $\eta$ is the identity. This means that the $\mathcal{A}$-ordering is the usual one. Since $N_i\in \Sigma_{\tau(i)}$ we get that the strict transform of $V(x,u_{\tau(i)})$ contains the rational curves $E_i,\ldots, E_d$. Hence the strict transform of $V(y,u_{\tau(i)})$ contains the rational curves $E_1,\ldots, E_{i-1}$. Given a $k>i$ the strict transform of $V(y,u_{\tau(k)})$ contains the rational curves $E_1,\ldots, E_{k-1}$. Hence the intersection of the strict transforms of $V(x,u_{\tau(i)})$ and $V(y,u_{\tau(k)})$ contains $E_i$ and therefore is nonempty. By Lemma~\ref{lem:empty} we have $\tau(i)<\tau(k)$ for every $i<k$, and we conclude that $\tau$ is the identity.
\end{proof}

\begin{Cor}
\label{cor:rational}
Let $\A$ be a smooth collection for $\{1,\ldots,d+1\}$. Let $\eta$ be a permutation of $1,\ldots,d+1$ such that $\eta(1)<_\A\eta(2)<_\A\ldots<_\A\eta(d+1)$. Then $E_i,\ldots, E_{d+1}$ is contained in the strict transform of $V(x,u_{\eta(i)})$ via $\phi\col\wt{T}_S^{\A}\to T_S$. Furthermore, the intersection of the strict transforms of the divisors 
\[
V(x,u_{\eta(1)}),\ldots,V(x,u_{\eta(i)}), V(y,u_{\eta(i)})\ldots, V(y,u_{\eta(d+1)})
\]
 is exactly the node $N_i$. 
\end{Cor}

\begin{Rem}
\label{rem:Weil}
Using Proposition~\ref{prop:smooth}, we see that for any given regular local family of curves without internal nodes $\pi\col\C\to B$ there exists a desingularizaion of $\C_S:=\C\times_B S$ obtained by blowing up Weil divisors. Moreover the map $\wt\pi_S\col\wt{\C}_S\to S$ is a regular family of curves, since $\wt{\C}_S$ is smooth. Also note that Corollary~\ref{cor:chain} implies that the fiber of $\wt\pi_S$ over the special point $0$ of $S$ is $C(d)$; more generally the fibers of $\wt\pi_S$ are either the smooth curve $\C_\eta$ over the generic point $\eta$ of $S$ or curves of the form $C(k)$ for some $k\in\{0,\ldots,d\}$. 
\end{Rem}

\section{Local conditions}

Recall that $B=\Spec(K[[t]])$, $S=\Spec(K[[u_1,\ldots, u_{d+1}]])$ and consider the map $S\to B$ given by  $t=u_1\cdot u_2\cdot\ldots\cdot u_{d+1}$. Let $\pi\col\C\to B$ be a pointed regular family of nodal curves with special fiber $C$ and section $\sigma\col B\to\C$ through the smooth locus of $\pi$. We let $P:=\sigma(0)$. Consider $\C_S:=\C\times_B S$ and let $\pi_S\col\C_S\to S$ be the induced map; form the fiber diagram
  \[
\begin{CD}
\C_S @>f>> \C \\
@V\pi_SVV   @VV\pi V\\
S @>>> B
\end{CD}
\]\smallskip
Any section $S\to\C_S$ of the map $\pi_S$ induces a $B$-map $S\to\C$ by composition; conversely, every $B$-map $S\to\C$ induces a section of $\pi_S$. We will abuse notation using the same name for both the section and the $B$-map. \par

 Let $\delta\col S\to \C$ be a $B$-map. Assume that $\delta(0)=N$, where $N$ is a node of $C$. We can write the completion of the local ring of $\C$ at $N$ as
\[
\widehat{\O}_{\C,N}\simeq K[[x,y]].
\]
The map $\pi\col\C\to B$ is given by $xy=t$ locally around $N$. Up to multiplication by an invertible element, the map $\delta$ is given by 
\begin{equation}
\label{eq:section}
x=u_A\quad\quad\text{and}\quad\quad y=u_{A^c},
\end{equation}
where $A$ is a proper nonempty subset of $\{1,\ldots,d+1\}$. Note that, geometrically, this means that $\delta(Q_j)\subset V(x)$ if and only if $j\in A$, where $Q_j$ is the generic point of $V(u_j)\subset S$.\par

 Given sections $\delta_1,\ldots,\delta_m$ of $\pi_S$ passing through nodes of $C$, a subcurve $Y$ of $C$ and a node $N$ of $C$, we define
\begin{equation}
\label{}
a_j^N(Y):=\#\left\{k\; |\; \delta_k(0)=N\quad\text{and}\quad \delta_k(Q_j)\subset Y^c\right\}.
\end{equation}
 Note that if $N\notin \Sigma_Y$, then the index $j$ plays no role; in this case we simply write $a^N(Y)$. Also note that if $N\in Y^{sing}$, then $a^N(Y)=0$.\par

Recall that $S_{\{j\}^c}$ is the complement of $\bigcup_{i\neq j} V(u_i)$ in $S$ and it is given by
\[
S_{\{j\}^c}=\Spec\left(K[[u_1,\ldots,u_{d+1}]]_{u_{\{j\}^c}}\right).
\]
 Hence there exists a map $S_{\{j\}^c}\to S$. Let $\C_{S_{\{j\}^c}}:=\C_S\times_S S_{\{j\}^c}$, and denote by $g_j\col\C_{S_{\{j\}^c}}\to\C_S$ the projection onto the first factor. Form the following fiber diagram
 \[
\begin{CD}
\C_{S_{\{j\}^c}} @>g_j>> \C_S \\
@V\pi_jVV   @VV\pi_S V\\
S_{\{j\}^c} @>>> S
\end{CD}
\]
Let $f_j:=f\circ g_j$ and $\delta_1,\ldots,\delta_m$ be sections of $\pi_S$ passing through nodes of $C$. If we restrict these sections to $S_{\{j\}^c}$, we obtain sections $S_{\{j\}^c}\to \C_{S_{\{j\}^c}}$ passing through the smooth locus of $\pi_j$.

Let $\L$ be a degree-$e$ invertible sheaf over $\C$. Denote by $\L_S$ the pullback of $\L$ to $\C_S$ and define
\[
\M:=\L_S\otimes f^*\O_{\C}(m\cdot\sigma(B))\otimes\I_{\delta_1(S)|\C_S}\otimes\cdots\otimes \I_{\delta_m(S)|\C_S}.
\]
Note that the sheaf $\M$ induces a rational map $S\dra \J$, since the generic fiber of $\pi_S$ is smooth. 
We also define 
\[
\M_j:=g_j^*\M,\quad\text{for every } j=1,\ldots, d+1.
\]
  Since the restricions of the sections $\delta_1,\ldots,\delta_m$ to $S_{\{j\}^c}$ are sections passing through the smooth locus of $C$, the sheaf $\M_j$ is invertible. Also, since $S_{\{j\}^c}$ is the spectrum of a DVR, there exists an invertible sheaf $\O_{\C_{S_{\{j\}^c}}}(-Z_j)$, where 
\[
Z_j=\sum_{i=1}^p \l_{i,j}\cdot f_j^*C_i,
\]
such that $M_j\otimes \O_{\C_{S_{\{j\}^c}}}(-Z_j)$ is $\sigma$-quasistable. \par
  Given a degree-$e$ invertible sheaf $\L$, sections $\delta_1,\ldots,\delta_m$ of $\pi_S$, a subcurve $Y$ of $C$ and a node $N$ in the intersection of $C_r$ and $C_s$, we define
\begin{equation}
b_j^N(Y,\L):=\left\{\begin{array}{ll}
                   \l_{r,j}-\l_{s,j} & \text{if}\;\; C_r\subset Y\;\text{and}\; C_s\not\subset Y   \\
                   \l_{s,j}-\l_{r,j} & \text{if}\;\; C_r\not\subset Y\;\text{and}\; C_s\subset Y   \\ 
                   0& \text{otherwise}
                   \end{array}\right.
\end{equation}
Recall that, since we are working with curves with no internal nodes, every node $N$ is external and hence there exists $r\neq s$ such that $N$ is in the intersection of $C_r$ and $C_s$.

\begin{Prop}
\label{prop:generic}
Let $\L$ be a degree-$e$ invertible sheaf on $\C$, let $\delta_1,\ldots,\delta_m$ be sections of $\pi_S$ passing through nodes of $C$ and $Y$ be a subcurve of $C$ containing $P$. Then, for every $h\in\{1,\ldots,d+1\}$ we have
\[
-\frac{k_Y}{2}<  \deg(\L|_Y)-e_Y+\sum _{N \in C^{sing}} (a^N_{h}(Y)-b^N_{h}(Y,L))\leq \frac{k_Y}{2}.
\]
\end{Prop}
\begin{proof}
Let $Q_h$ be the generic point of $V(u_h)$. Identify $Y$ with $Y\times_B Q_h$. By the definition of $a^N$ and by the fact that each section $\delta_i$ goes through some node $N$, we clearly have
\[
\deg(\M_j|_Y)=\deg(\L|_Y)+\sum_{N\in C^{sing}}a^N_{h}(Y).
\]
 Indeed, the sum $\sum a^N_{h}(Y)$ is the number of sections $\delta_i$ such that $\delta_i(Q_h)\in Y^c$. It follows that 
\[
\deg(\M_j\otimes\O_{\C_{S_{\{h\}^c}}}(-Z_h)|_Y)=\deg(L|_Y)+\sum_{N\in C^{sing}}a^N_h(Y)-Z_h\cdot Y.
\]
However we have
\[
Z_h\cdot Y=\sum_{N\in C^{sing}}b^N_h(Y,L),
\]
which concludes the proof.
\end{proof}
 
 Let $\phi\colon\wt{\C}_S\to\C_S$ be a fixed desingularization of $\C_S$ as in Remark~\ref{rem:Weil}. Let $\wt{\Delta}_j$ be the strict transform of $\delta_j(S)$. Since $\wt{\C}_S$ is regular, it follows that $\wt{\Delta}_j$ is a Cartier divisor. We define
 \[
 \wt{C}:=(\pi_S\circ\phi)^{-1}(0),
 \]
 the special fiber of the map $\pi_S\circ\phi$. Recall that, by Proposition~\ref{prop:smooth}, we have an identification of $\wt C$ with $C(d)$. Let $\C_{[u_j\;i]}$ be the closure of $\wt{g}_j(f_j^{-1}(C_i))$ in $\wt{\C}_S$, where $\wt{g}_j$ is the induced map $\wt{g}_j:\C_{S_{\{j\}^c}}\to\wt{\C}_S$ and $f_j=f\circ g_j$. Finally, we let 
 \begin{equation}
 \label{eq:zj}
 \overline{Z}_j:=\sum_{i=1}^p \ell_{i,j}\cdot \C_{[u_j\;i]},
 \end{equation}
and define the invertible sheaf on $\wt\C_S$
 \begin{equation}
 \label{eq:mt}
\wt{\M}_{\phi}:=\phi^*(\L_S)\otimes\phi^*f^*(\O_\C(m\cdot\sigma(B)))\otimes\O_{\wt{\C}_S}\left({-\sum_{j=1}^m\wt{\Delta}_j}-\sum_{j=1}^{d+1}\overline{Z}_j\right).
\end{equation}

\begin{Thm}
\label{thm:map}
Let $\L$ be a degree-$e$ invertible sheaf on $\C$ and $\delta_1,\ldots,\delta_m$ be sections of $\pi_S$. There exists a map $S\to \Jb$ extending the rational map defined by $\M$ if the following two conditions hold for every subcurve $Y\subset C$ containing $P$\smallskip

\begin{enumerate}
\item For every $j_1,j_2=1,\ldots,d+1$ and every node $N\in\Sigma_Y$, we have
\[
|(a^N_{j_1}(Y)-b^N_{j_1}(Y,L))-(a^N_{j_2}(Y)-b^N_{j_2}(Y,L))|\leq 1.
\]

\item For every function $j\col C^{sing}\to \{1,\ldots,d+1\}$, we have
\[
-\frac{k_Y}{2}< \deg(\L|_Y)-e_Y+\sum _{N \in C^{sing}} (a^N_{j(N)}(Y)-b^N_{j(N)}(Y,L)) \leq \frac{k_Y}{2}.
\]
\end{enumerate}
In this case, if $\phi:\wt{C}_S\to\C_S$ is any desingularization of $\C_S$ as in Remark~\ref{rem:Weil} then this map is induced by the invertible sheaf $\wt{\M}_{\phi}$.
\end{Thm}

\begin{proof}
Throughout the proof we fix a desingularization $\phi:\wt{C}_S\to \C_S$ of $\C_S$ as in Remark~\ref{rem:Weil} and set $\wt{\M}:=\wt{\M}_{\phi}$.
It is enought to prove that under the hypothesis the sheaf $\phi_*(\wt{\M})$ is $\sigma$-quasistable. By Proposition~\ref{prop:pushforward} it suffices to check that $\wt{\M}$ is admissible and $\sigma$-quasistable over each connected subcurve $Y$ of $\wt{C}$ with connected complement such that $Y$ and $Y^c$ are not contracted by the map $\phi$.\par
  We begin by computing the degrees of the restriction of $\wt{\M}$ to the components of the special fiber. Let $E$ be a chain of $\phi$-exceptional components and let $N=\phi(E)$. Clearly the degree of 
  \[
  \phi^*(\L_S)\otimes\phi^*f^*(\O_\C(m\cdot\sigma(B)))|_E
  \]
  is zero. The same property holds for $\O_{\wt{C}_S}(-\wt{\Delta}_j)|_E$ if the section $\delta_j$ does not pass through the node $N$. Since $E$ is contracted we can look locally around the node $N$. Let $T:=\Spec{\hat{\O}_{C,N}}$; the subcurve $E$ can be seen as a subcurve of the special fiber of the map $\wt{T}_S^\A\to S$, for some collection $\A$, as in Section~\ref{sec:des}. Let $\{N_i,N_j\}= E\cap E^c$ be the extremal nodes of the chain. It follows from Proposition~\ref{prop:node} that $N_i\in\Sigma_{\eta(i)}$ and $N_j\in\Sigma_{\eta(j)}$. Without loss of generality we will assume that $\eta$ is the identity. \par
  Let $Y$ be a fixed subcurve of $C$ containing $P$ and admitting $N$ as an extremal node. Let $\delta_k$ be a section through $N$. Up to renaming $i$ and $j$, we can assume that the strict transform of $Y\times V(u_{i})$  does not contain the node $N_j$. Let $Q_i$ be the generic point of $V(u_i)$.  Set $Y_i:=Y\times Q_i\subset \wt{\C}_S$ and $Y_{i,0}:=\overline{Y}_i\cap \wt{C}$, where the bar denotes the closure in $\wt{\C}_S$.\par
   The degree of $\O_{\wt\C_S}(-\wt{\Delta}_k)|_{Y_{i}}$ is $-1$ if $\delta_k(Q_{i})\in Y_{i}$, and $0$ otherwise. Since the degree of $\O_{\wt\C_S}(-\wt{\Delta}_k)|_{Y_{i,0}}$ is the same as the degree of $\O_{\wt\C_S}(-\wt{\Delta}_k)|_{Y_i}$, then we have 
  \[
  \deg(\O_{\wt\C_S}(-\wt{\Delta}_k)|_E)=\left\{\begin{array}{rl}
                                           1 & \quad\text{if $\delta_k(Q_i)\in Y_i$ and $\delta_k(Q_j)\notin Y_j$,} \\
                                           -1 & \quad\text{if $\delta_k(Q_i)\notin Y_i$ and $\delta_k(Q_j)\in Y_j$,} \\
                                           0 & \quad\text{otherwise.}
                                           \end{array}\right.
   \]
   Indeed, notice that $\overline{Y_{j,0}\setminus Y_{i,0}}\cup \overline{Y_{i,0}\setminus Y_{j,0}}$ consists of $E$ and other chains of rational curves contracted by $\phi$; since the section $\delta_k$ goes through the node $N$, it follows that the line bundle $\O_{\wt\C_S}(-\wt{\Delta}_k)$ restricted to these other chains has degree $0$.
   Similarly, we can compute: 
\[
  \deg(\O_{\wt\C_S}(-\C_{[u_k\;r]}|_E))=\left\{\begin{array}{rl}
                                           1 & \quad\text{if $k=j$ and $C_r\subset Y$,}\\
                                           1 & \quad\text{if $k=i$ and $C_r\subset Y^c$,} \\
                                           -1& \quad\text{if $k=j$ and $C_r\subset Y^c$,}\\
                                           -1& \quad \text{if $k=i$ and $C_r\subset Y$,}\\
                                           0 & \quad\text{otherwise.}
                                           \end{array}\right.
   \]
Summing up all the contributions, we get that the degree of $\wt{\M}|_E$ is 
\[
(a^N_{j}(Y)-b^N_{j}(Y,L))-(a^N_{i}(Y)-b^N_{i}(Y,L))
\]
and this shows that the admissibility of $\wt{\M}$ is equivalent to condition (1).\par

     Let $Z$ be a connected subcurve of $\wt{C}$ with connected complement such that neither $Z$ nor $Z^c$ are contracted by $\phi$ and set $Y:=\phi(Z)\subset C$. We can assume that $P\in Y$. We want to compute the degree of $\wt{\M}|_Z$. Again by Proposition~\ref{prop:node} each extremal node of $Z$ belongs to only one $\Sigma_j$. Let $j_Z:\Sigma_Z\to\{1,\ldots,d+1\}$ be the induced function; note that the extremal nodes of $Z$ map bijectively onto the extremal nodes of $Y$ and hence we can also consider $\Sigma_Y$ as a domain for the function $j_Z$.
     We have
\[
\deg(\phi^*(\L_S)\otimes\phi^*f^*(\O_\C(m\cdot\sigma(B)))|_Z)=\deg(\L|_Y)+m.
\]
If we fix $k$ such that $\delta_k(0)=N$, we also have 
\[
  \deg(\O_{\wt\C_S}(-\wt{\Delta}_k)|_Z)=\left\{\begin{array}{rl}
                                           -1 & \quad\text{if $N\in \Sigma_Y$ and $\delta_k(Q_{j_Z(N)})\in Y_{j_Z(N)}$}\\
                                           -1 & \quad\text{if $N\in Y\setminus\Sigma_Y$}\\
                                           0 & \quad\text{otherwise.}
                                           \end{array}\right.
\]    
Moreover, we have
\[
  \deg(\O_{\wt\C_S}(-\C_{[u_k\;r]})|_Z)=\left\{\begin{array}{rl}
                                           -\#(j_Z^{-1}(k)\cap\Sigma_{C_r}) & \quad\text{if $C_r\subset Y^c$}\\
                                           \#(j_Z^{-1}(k)\cap\Sigma_{C_r}) & \quad\text{if $C_r\subset Y$}.\\
                                           \end{array}\right.
\]      
Indeed, $j_Z^{-1}(k)$ is the collection of extremal nodes of $Z$ that belong to $\Sigma_k$. This means that over each node $N\in \Sigma_Y\setminus j_Z^{-1}(k)$ the intersection of the divisor $\C_{[u_k\; r]}$ with $Z$ is either empty or a chain of rational curves contracted by $\phi$ with image the node $N$. In either case the contribution to the intersection number $\C_{[u_k\;r]}\cdot Z$  is zero. On the other hand if $N \in j_Z^{-1}(k)$, then the intersection $\C_{[u_k\; r]}\cap Z\cap\phi^{-1}(N)$ is a single point, and in this case the contribution to the intersection number is 1. The case where $C_r\subset Y$ is analogous. 
  To sum up, if we define 
 \begin{equation}
 \label{eq:c}
 \begin{array}{rl}
  c:=&\displaystyle\sum_{N\in \Sigma_Y}\#\{k\; |\;\delta_k(0)=N\;\text{and}\;\delta_k(Q_{j_Z(N)})\in Y_{j_Z(N)}\}+\\
  &+\displaystyle\sum_{N\in Y^{sing}}\#\{k\;|\;\delta_k(0)=N\},
  \end{array}
 \end{equation}
 then the degree of $\wt{\M}|_Z$ is
 \[
 \deg(\L|_Y)+m-c+\sum_{k=1}^{d+1}\sum_{r=1}^p\ell_{r,k}\deg(\O_{\wt\C_S}(-\C_{[u_k\;r]})|_Z).
 \]

We note now that 
\begin{equation}
\label{eq:mca}
m-c=\sum_{N\in \Sigma_Y}a^N_{j_Z(N)}(Y)+\sum_{N\in C^{sing}\setminus\Sigma_Y}a^N(Y).
\end{equation}
In fact, we have a total of $m$ sections, and hence $m-c$ is the number of sections that do not satisfy the conditions in Equation~\ref{eq:c}, i.e. the number of sections that satisfy either $\delta_k(0)\in \Sigma_Y$ and $\delta_k(Q_{j_Z(N)})\in Y^c$ or $\delta_k(0)\in (Y^c)^{sing}$ and $\delta_k(Q_{j_Z(N)})\in Y^c$. This is clearly equal to the right hand side of Equation~\ref{eq:mca}.\par
  
  Let $\epsilon_r$ be $1$ if $C_r\subset Y$ and $-1$ otherwise.  We have
 \begin{eqnarray*}
 \sum_{k=1}^{d+1}\sum_{r=1}^p\ell_{r,k}\deg(\O_{\wt\C_S}(-\C_{[u_k\;r]})|_Z)  &=  &
 \sum_{k=1}^{d+1}\sum_{r=1}^p\underset{N\in C_r}{\sum_{j_Z(N)=k}}\epsilon_r \ell_{r,k} \\
  &=&\sum_{N\in C^{sing}}\underset{r=1,\ldots,p}{\sum_{N\in C_r}}\epsilon_r\ell_{r,j_Z(N)}
 \end{eqnarray*}
and, since $N$ only belongs to two components we also have 
\[
\underset{r=1,\ldots,p}{\sum_{N\in C_r}}\epsilon_r\ell_{r,j_Z(N)}=-b^Y_{j_Z(N)}(Y,\L).
\]
Therefore, we conclude that 
\begin{eqnarray*}
\sum_{k=1}^{d+1}\sum_{r=1}^p\ell_{r,k}\deg(\O_{\wt\C_S}(-\C_{[u_k\;r]})|_Z)&=&-\sum_{N\in\Sigma_Y}b^Y_{j_Z(N)}(Y,\L)\\
&=&-\sum_{N\in C^{sing}}b^Y_{j_Z(N)}(Y,\L)
\end{eqnarray*}
and the proof is complete.
\end{proof}

\section{Curves with two components}

Let $\pi\colon\C\to B$ be a pointed smoothing of a nodal curve $C$ with section $\sigma\colon B\to \C$ through the smooth locus of $\pi$. Let $\L$ be a invertible sheaf of degree $e$ over $\C$. From now on, we assume that $C$ has two smooth components $C_1$ and $C_2$ meeting at $q$ nodes $N_1,\ldots, N_q$, with the marked point $P:=\sigma(0)$ on the component $C_1$.  Locally around each node $N_\l$, the completion of the local ring of $\C$ at $N_\l$ is given by
\[
\wh{\O}_{\C,N}\simeq K[[x,y]],
\]
where $x=0$ is the local equation of $C_1$ and $y=0$ is that of $C_2$. Hence, if we let $T=\Spec(K[[x,y]])$, then, for each node $N$, there exists a map $T\to \C$ taking the closed point of $T$ to $N$. Moreover, the composition map $T\to B$ is given by $t=xy$.\par

  Our goal is to resolve the rational map $\alpha^d_\L\colon\C^d\dashrightarrow \Jb$. Let $\wt\C^d$ be the desingularization of $\C^d$ obtained inductively as follows:
   First define $\wt\C^1:=\C^1$. Then assume that the desingularization $\wt\C^d$ of  $\C^d$ is given and let $\wt{\C}^{d+1}\to\wt{\C}^d\times_B\C$ be the sequence of blowups  along the strict transforms of the following Weil divisors in the stated order
   \begin{equation}
   \label{eq:Delta}
   \Delta_{d,d+1},\Delta_{d-1,d+1},\ldots,\Delta_{1,d+1},
   \end{equation}
   and then
   \begin{equation}
   \label{eq:C}
   C_1^{d+1}, C_1^{d}\times C_2,C_1^{d-1}\times C_2\times C_1, C_1^{d-1}\times C_2^2, \ldots, C_2^{d-1}\times C_1\times C_2,C_2^{d}\times C_1,C_2^{d+1},
      \end{equation}
      where $\Delta_{i,d+1}$ is the image of the section $\wt{\C}^d\to\wt{\C}^d\times_B\C$, induced by the composition $\delta_i\col\wt{\C}^d\to\C^d\to \C$, where the last map is the projection onto the $i$-th factor.\par
      
      \begin{Lem}
      The scheme $\wt{\C}^d$ is smooth. 
      \end{Lem}
      \begin{proof}
      We proceed by induction. Of course, $\wt{\C}^1$ is smooth.  Note that $\wt{\C}^d$ is given locally by $\Spec(K[[u_1,\ldots, u_{d+1}]])$. Moreover, the map $\wt{\C}^{d}\to B$ is given by $t=u_1\cdot\ldots \cdot u_k$ and each $u_j$ is the local equation of the stric transform $\C_{[u_j]}$ of some $C_{\epsilon_1}\times\ldots\times C_{\epsilon_d}$ via the map $\wt{\C}^d\to\C^d$, for $j=1,\ldots,k$ and some $\epsilon_1,\ldots,\epsilon_d\in\{1,2\}$. 
      We may assume that $k=d+1$, the other cases being analogous. Following the notation in Section~\ref{sec:des}, we see that $S$ is the local description of $\wt{\C}^d$ and $T$ the one of $\C$. Moreover, the collection $\A$ induced by the sequence of blowups \eqref{eq:Delta} and \eqref{eq:C} is smooth, because there exists a set $A_\ell\in \A$ with only $j$ as an element for each $j$. Therefore, by Proposition~\ref{prop:smooth}, $\wt{\C}^d$ is smooth.
      \end{proof}
     
     Let $\phi:\wt{\C}^{d+1}\to\wt{\C}^d\times_B \C$ be the desingularization given above. The projection $\wt{\pi}:\wt{\C}^{d+1}\to \wt{\C}^d$ onto the first factor is a regular family of nodal curves. 
     As in \eqref{eq:mt}, we define the sheaf $\wt{\M}$ on $\wt{\C}^{d+1}$ as
     \[
     \wt\M:=\phi^*f^*(\L\otimes\O_{\C}(d\cdot \sigma(B))\otimes \O_{\wt{\C}^{d+1}}\left(-\sum_{i=1}^{d}\wt{\Delta}_{i,d+1}\right)\otimes \O_{\wt{\C}^{d+1}}(-\Z)
     \] 
where $\Z$ is defined as the sum of the strict transforms of the divisors $Z_{\underline{i}}$, defined in Equation~\eqref{eq:z}, via the map $\phi$. 
    
    \medskip
    
    We can now state our main result.
    
    \begin{Thm}
    \label{thm:main}
    There exists a map $\ol{\alpha}^d_\L\colon\wt{\C}^d\to \Jb$ induced by $\wt{\M}$ extending the map $\alpha^d_\L$.
      \end{Thm}
      
      We devote the rest of this section to prove Theorem \ref{thm:main}.
      
\subsection{Special points}
    
    Define $U_i$, for $i=1,\ldots, d$, as the locus in $\wt{\C}^d$ such that the fiber of $\wt{\pi}$ is the curve $C(i)$ (see Corollary~\ref{cor:chain} for a description of these loci), and $U_0:=\dot\C^d$. Then $U_0,\ldots, U_d$ define a stratification of $\wt{\C}^d$ by locally closed subschemes. Furthermore, note that each neighborhood of $U_d$ intersects every irreducible component of $U_i$. Note that the sheaf $\wt\M$ is $\sigma$-quasistable over $U_0$.\par
    
     In order to prove the theorem we can argue locally on the base. In fact it suffices to prove that there exists a map extending $\alpha^d_\L$ locally around $U_d$, since $\sigma$-quasistability is an open condition by \cite[Prop. 34]{E01} and the restrictions of $\wt\M$ to the fibers over points in a connected component of $U_i$ have constant multidegree.
     
     Let $R$ be a point in $U_d$. We call the point $R$ a \emph{special point}. Locally around $R$, the scheme $\wt{\C}^d$ is given by $S_R=\Spec(K[[u_1,\ldots,u_{d+1}]])$. Let $\iota_R\col S_R\to \wt{\C}^d$ be the natural map. Let also $S_R\to B$ be the restriction of the map $\wt\C^d\ra B$; hence $S_R\to B$ is given by $t=u_1\cdot\ldots\cdot u_{d+1}$. 
     
      We can associate to the special point $R$ a $d$-tuple $(\ell_1,\ell_2,\ldots,\ell_d)$, with $\ell_k\in\{1,\ldots, q\}$, by the rule $\delta_k(R)=N_{\ell_k}$. Also, we can associate to $u_j$ a $d$-tuple  $[\epsilon_1\;\ldots\;\epsilon_d]$ with $\epsilon_j\in\{1,2\}$, where $u_j$ is the local equation of the strict transform $\C_{[u_j]}$ of $C_{\epsilon_1}\times\ldots\times C_{\epsilon_d}$ via $\wt{\C}^d\to \C^d$. Abusing notation we will denote this $d$-tuple by $[u_j]$ and we set $u_j(k):=\epsilon_k$. We define a \emph{special point data} as a set
      \[
      \mathcal{R}:=\{(\l_1,\ldots,\l_d), [u_1],\ldots, [u_{d+1}]\}.
      \]
We call such a data a \emph{constructible special point data} if it arises from a special point $R$ of $\wt{C}^d$. In this case we may simply refer to it as a \emph{special point} and denote it by $R$. \par
  Let $R=\{(\ell_1,\ldots,\ell_d),[u_1],\ldots, [u_{d+1}]\}$ be a special point of $\wt{\C}^d$ and $N_\ell$ be a node of $C$. We have maps $S:=S_R\to \wt{\C}^d$ and $T\to\C$ associated to these points. Then Equations \eqref{eq:Delta} and \eqref{eq:C} induce a collection $\A_{R,\ell}$ of subsets of $\{1,\ldots, d+1\}$ that gives the desingularization of $T_S$, as in Equation \eqref{eq:blowup}. We will call the $\A_{R,\ell}$-ordering of $[u_1],\ldots,[u_{d+1}]$ simply the \emph{$\ell$-ordering of $[u_1],\ldots,[u_{d+1}]$}.

    We proceed now to determine what special point data are constructible. For $d=1$ the only constructible special point data are of the form
    \[
    \{(\l),[1],[2]\}.
    \]
    
For $d=2$, we use Corollary~\ref{cor:rational}. We just need to find each collection $\A_{R,\ell}$ associated with the blowup described by Equations \eqref{eq:Delta} and \eqref{eq:C}. First, note that each special point $\{(\l_1),[1],[2]\}$ in $\wt{\C}^1$ is locally given by $t=[1]\cdot[2]$, and each node $N_{\l_2}$ of $\C$ is given locally by $t=xy$. Therefore, we just need to compute the local equations of the diagonal of $\C^2$ and of each one of the divisors $C_1\times C_1$, $C_1\times C_2$, $C_2\times C_1$ and $C_2\times C_2$. If $\l_2\neq \l_1$ then the diagonal is empty, otherwise the equation of the diagonal is $(x-[1],y-[2])$. On the other hand, the equation of $C_\epsilon\times C_1$ is $(x,[\epsilon])$ and of $C_{\epsilon}\times C_2$ is $(y,[\epsilon])$, for $\epsilon\in\{1,2\}$. By Remark~\ref{rem:diagonal}, the blowup of the diagonal is locally given by the blowup of $V(x,[2])$, while the blowup of $C_{\epsilon}\times C_2$ is locally given by the blowup of $V(x,[3-\epsilon])$. It follows that the $\ell_2$-ordering of $[1],[2]$ is 
\begin{eqnarray*}
[1],[2] & \text{if} & \l_2\neq\l_1\\ 
 \;[2] ,[1] & \text{if} & \l_2 = \l_1.
\end{eqnarray*}
 Note that in Corollary~\ref{cor:rational} the nodes $N_i$ are the special points. Moreover the strict transform of the divisor $V(x,[\epsilon])$ becomes $[\epsilon\; 1]$ and that of $V(y,[\epsilon])$ becomes $[\epsilon\; 2]$. Therefore, the contructible special point data for $\wt{\C}^2$ are
\[
\{(\l,\l),[21],[22],[12]\}\quad\text{and}\quad \{(\l,\l),[21],[11],[12]\}
\]
and 
\[
\{(\l_1,\l_2),[11],[12],[22]\}\quad\text{and}\quad\{(\l_1,\l_2),[11],[21],[22]\},
\]
for $\l_1\neq\l_2$.

For $d=3$, we proceed in a similar fashion.  First, fix a special point $R$ with special point data $\{(\l_1,\l_2),[u_1],[u_2],[u_3]\}$ in $\wt{\C}^2$ and choose a node $N_{\l_3}$ of $C$. The equation of the diagonal $\Delta_{k,3}$ is of the form
\[  
(x-u_{A'_k},y-u_{(A'_k)^c}),
\]
where 
\[
A'_k:=\{j\; |\; u_j(k)=1\quad\text{and}\quad \l_3=\l_k\},
\]
since $x=0$ is the equation of $C_1$. Note that if $\l_3\neq\l_k$ then $A'_k$ is empty and then so is $\Delta_{k,3}$. On the other hand, the local equation of $C_{[u_j]}\times C_{1}$ is $(x,u_j)$ and the one of  $C_{[u_j]}\times C_2$ is $(y,u_j)$. It follows that the blowup of the diagonal $\Delta_{k,3}$ is locally given by the blouwup of $V(x,u_{(A'_k)^c})$ and the blowup of $C_{[u_j]}\times C_2$ is locally given by the blowup of $V(x,u_{\{j\}^c})$.\par 
  Now, for $R=\{(\l,\l),[11],[12],[21]\}$ and $\l_3=\l$, we see that the collection $\A_{R,\ell_3}$ is given by $A_1=(A'_2)^c=\{[12]\}$, $A_2=(A'_1)^c=\{[21]\}$, $A_3=\{[11]\}$, $A_4=\{[12],[21]\}$, and so on. However $(A_1,A_2)$ is a smooth collection. Then the $\ell_3$-ordering of $[11],[12],[21]$ is $[12],[21],[11]$, and hence, by Corollary~\ref{cor:rational}, we have $3$ special points in $\wt{\C}^3$ lying over $R$, that are
\begin{equation}
\label{eq:c3}
\begin{array}{l}
\{(\l,\l,\l),[121],[122],[212],[112]\}\\
\{(\l,\l,\l),[121],[211],[212],[112]\}\\
\{(\l,\l,\l),[121],[211],[111],[112]\}.
\end{array}
\end{equation}
Similarly, for $R=\{(\l,\l),[12],[21],[22]\}$ and $\l_3=\l$, the $\ell_3$-ordering of $[12]$, $[21]$, $[22]$ is $[22],[12],[21]$, then we get the 3 special points
\[
\begin{array}{l}
\{(\l,\l,\l),[221],[222],[122],[212]\}\\
\{(\l,\l,\l),[221],[121],[122],[212]\}\\
\{(\l,\l,\l),[221],[121],[211],[212]\}.
\end{array}
\]
As for the case $R=\{(\l,\l),[11],[12],[21]\}$ and $\l_3\neq\l$, we see that the two diagonals are empty. Therefore $A_1=\{[11]\}$, $A_2=\{[12],[21]\}$, $A_3=\{[12]\}$, $A_4=\{[11],[21]\}$, and so on. We see that $(A_1,A_2,A_3)$ is a smooth collection, and in fact the given desingularization is the same as that given by the collection $(A_1,A_3)$. The $\ell_3$-ordering of $[11],[12],[21]$ is $[11],[12],[21]$. It follows again from Corollary~\ref{cor:rational} that the special points of $\wt{\C}^3$ over $R$ are
\[
\begin{array}{l}
\{(\l,\l,\l_3),[111],[112],[122],[212]\}\\
\{(\l,\l,\l_3),[111],[121],[122],[212]\}\\
\{(\l,\l,\l_3),[111],[121],[211],[212]\}.
\end{array}
\]
Similarly, for $R=\{(\l,\l),[12],[21],[22]\}$ and $\l_3\neq \l$, the $\ell_3$-ordering of $[12]$, $[21]$, $[22]$ is $[12],[21],[22]$, and then we get the 3 special points
\[
\begin{array}{l}
\{(\l,\l,\l_3),[121],[122],[212],[222]\}\\
\{(\l,\l,\l_3),[121],[211],[212],[222]\}\\
\{(\l,\l,\l_3),[121],[211],[221],[222]\}.
\end{array}
\]
As for the case $R=\{(\l_1,\l_2),[11],[12],[22]\}$ with $\l_1\neq\l_2$ and $\l_3=\l_1$, we see that the diagonal $\Delta_{2,3}$ is empty. Therefore $A_1=\{[22]\}$ (it comes from the diagonal $\Delta_{1,3}$), $A_2=\{[11]\}$, $A_3=A_2^c$, $A_4=\{[12]\}$, and so on. We see that $(A_1,A_2)$ is a smooth collection. The $\ell_3$-ordering of $[11],[12],[22]$ is $[22],[11],[12]$. It follows again from Corollary~\ref{cor:rational} that the special points of $\wt{\C}^3$ over $R$ are
\[
\begin{array}{l}
\{(\l_1,\l_2,\l_1),[221],[222],[112],[122]\}\\
\{(\l_1,\l_2,\l_1),[221],[111],[112],[122]\}\\
\{(\l_1,\l_2,\l_1),[221],[111],[121],[122]\}.
\end{array}
\]
Similarly for $R=\{(\l_1,\l_2),[11],[21],[22]\}$ with $\l_1\neq\l_2$ and $\l_3= \l_1$, the $\ell_3$-ordering of $[11],[21],[22]$ is $[21],[22],[11]$, and then we get the special points
\[
\begin{array}{l}
\{(\l_1,\l_2,\l_1),[211],[212],[222],[112]\}\\
\{(\l_1,\l_2,\l_1),[211],[221],[222],[112]\}\\
\{(\l_1,\l_2,\l_1),[211],[221],[111],[112]\}.
\end{array}
\]
As for the case $R=\{(\l_1,\l_2),[11],[12],[22]\}$ with $\l_1\neq\l_2$ and $\l_3=\l_2$, we see that the diagonal $\Delta_{1,3}$ is empty. Therefore $A_1=\{[12],[22]\}$ (it comes from the diagonal $\Delta_{2,3}$), $A_2=\{[11]\}$, $A_3=A_2^c$, $A_4=\{[12]\}$, and so on. We see that
 $(A_1,A_2,A_3,A_4)$ is a smooth collection, and in fact the given desingularization is the same as that given by the collection $(A_1,A_4)$. The $\ell_3$-ordering of $[11],[12],[22]$ is $[12],[22],[11]$. It follows again from Corollary~\ref{cor:rational} that the special points of $\wt{\C}^3$ over $R$ are
\[
\begin{array}{l}
\{(\l_1,\l_2,\l_2),[121],[122],[222],[112]\}\\
\{(\l_1,\l_2,\l_2),[121],[221],[222],[112]\}\\
\{(\l_1,\l_2,\l_2),[121],[221],[111],[112]\}.
\end{array}
\]
Similarly for $R=\{(\l_1,\l_2),[11],[21],[22]\}$ with $\l_1\neq\l_2$ and $\l_3= \l_2$, the $\ell_3$-ordering of $[11],[21],[22]$ is $[22],[11],[21]$, and then we get special points
\[
\begin{array}{l}
\{(\l_1,\l_2,\l_2),[221],[222],[112],[212]\}\\
\{(\l_1,\l_2,\l_2),[221],[111],[112],[212]\}\\
\{(\l_1,\l_2,\l_2),[221],[111],[211],[212]\}.
\end{array}
\]
As for the case $R=\{(\l_1,\l_2),[11],[12],[22]\}$ with $\l_1\neq\l_2$ and $\l_3\neq\l_1,\l_2$, we see that the diagonals are empty. Therefore $A_1=\{[11]\}$, $A_2=A_1^c$, $A_3=\{[12]\}$, $A_4=A_3^c$, and so on. We see that $(A_1,A_2,A_3)$ is a smooth collection, and in fact the given desingularization is the same as that given by the collection $(A_1,A_3)$. The $\ell_3$-ordering of $[11],[12],[22]$ is $[11],[12],[22]$. It follows from Corollary~\ref{cor:rational} that the special points of $\wt{\C}^3$ over $R$ are
\[
\begin{array}{l}
\{(\l_1,\l_2,\l_3),[111],[112],[122],[222]\}\\
\{(\l_1,\l_2,\l_3),[111],[121],[122],[222]\}\\
\{(\l_1,\l_2,\l_3),[111],[121],[221],[222]\}.
\end{array}
\]
Similarly for $R=\{(\l_1,\l_2),[11],[21],[22]\}$ with $\l_1\neq\l_2$ and $\l_3\neq \l_1,\l_2$, the $\ell_3$-ordering of $[11],[21],[22]$ is $[11],[21],[22]$, and then we get special points
\[
\begin{array}{l}
\{(\l_1,\l_2,\l_3),[111],[112],[212],[222]\}\\
\{(\l_1,\l_2,\l_3),[111],[211],[212],[222]\}\\
\{(\l_1,\l_2,\l_3),[111],[211],[221],[222]\}.
\end{array}
\]

Let $R=\{(\l_1,\ldots,\l_d),[u_1],\ldots,[u_{d+1}]\}$ be a special point of $\wt{\C}^d$ and a let $N_{\l_{d+1}}$ be a node of $C$. The special points of $\wt{\C}^{d+1}$ over $(R, N_{\l_{d+1}})\in\wt{\C}^d\times_B\C$ are of the form 
\begin{equation}
\label{eq:special}
\{(\l_1,\ldots,\l_d,\l_{d+1}),[v_1\; 1], [v_2\; 1],\ldots, [v_h\; 1], [v_h \;2],\ldots, [v_{d+1}\; 2]\},
\end{equation}
where $[v_1],[v_2],\ldots, [v_{d+1}]$ is the $\ell_{d+1}$-ordering of $[u_1],\ldots, [u_{d+1}]$, for each $h=1,\ldots, d+1$. Recall that $\A_{R,\ell_{d+1}}$ is the collection associated to the blowup given by Equations \eqref{eq:Delta} and \eqref{eq:C}. As in the case $d=3$, we see that the equation of the diagonal $\Delta_{k, d+1}$ is of the form
\[
(x-u_{A'_k},y-u_{(A'_k)^c})
\]
where 
\[
A'_k=\{j \;|\; u_j(k)=1\quad\text{and}\quad \l_{d+1}=\l_k\}.
\]
In particular if $\l_{d+1}\neq\l_k$, then $A'_k$ is empty and so is $\Delta_{k, d+1}$. If $[u_1],\ldots, [u_{d+1}]$ is written in lexicographical order, then $A_1=(A'_d)^c$, $A_2=(A'_{d-1})^c$, $\ldots$, $A_d=(A'_1)^c$, $A_{d+1}=\{[u_1]\}$, $A_{d+2}=\{[u_1]\}^c$, $A_{d+3}=\{[u_2]\}$, and so on. Note that some of $A_1$,\ldots, $A_{d}$ might be empty, and in the examples above we omitted such sets. We sum-up what we have shown in the following lemma.

\begin{Lem}
\label{lem:ordering}
We have $u_{j_1}<_{\ell_{d+1}} u_{j_2}$ if and only if one of the following conditions holds.
\begin{enumerate}
\item There exists $k_0$ such that $u_{j_1}(k_0)=2$ and $u_{j_2}(k_0)=1$ with $\l_{k_0}=\l_{d+1}$; moreover, $u_{j_1}(k)=u_{j_2}(k)$ for each $k>k_0$ such that $\l_{k}=\l_{d+1}$.
\item For all $k$ such that $\l_k=\l_{d+1}$ we have $u_{j_1}(k)=u_{j_2}(k)$ and there exists $k_0$ such that $\l_{k_0}\neq \l_{d+1}$ with $u_{j_1}(k_0)=1$ and $u_{j_2}(k_0)=2$; moreover, for all $k<k_0$ we have $u_{j_1}(k)=u_{j_2}(k)$.
\end{enumerate}
\end{Lem}

\subsection{Proof of the main theorem}

We now start to check the conditions in Theorem~\ref{thm:map}. Given a special point $R=\{(\l_1,\ldots,\l_d),[u_1],\ldots,[u_{d+1}]\}$ in $\wt{\C}^d$, we only need to compute the numbers $a^N_j(C_1)$ and $b^N_j(C_1,\L)$, where $N$ is a node of $C$ and $j=1,\ldots,d+1$. We observe that $a^N_j(C_1)$ is the number of sections $\delta_k$ such that $\delta_k(R)=N$ and $\delta_k(Q_j)\in C_2$, where $Q_j$ is the generic point of $V(u_j)$. Thus, by the definition of $\delta_k$, we see that $a^N_j(C_1)$ is the number of $k$'s such that $\delta_k(R)=N$ and $u_j(k)=2$. Thus, it is convenient to define
\begin{eqnarray*}
a^\l_{[u_j],R}&:=&a^{N_\l}_j(C_1)=\#\{k\; |\; \ell_k=\l\quad\text{and}\quad u_j(k)=2\},\\
a_{[u_j],R}&:=&(a^1_{[u_j],R},a^2_{[u_j],R},\ldots,a^q_{[u_j],R}),\\
|a_{[u_j],R}|&:=&\sum_{\l=1}^qa^\l_{[u_j],R}.
\end{eqnarray*}
It is easy to check that
\begin{equation}
\label{eq:b}
b_{u_j,R}:=b^N_j(C_1,\L)=\left\lceil\frac{|a_{[u_j],R}|-\deg(L|_{C_2})+e_{C_2}}{q}-\frac{1}{2}\right\rceil.
\end{equation}

\begin{Prop}
\label{prop:main}
Let $R=\{(\l_1,\ldots,\l_d),[u_1],\ldots, [u_{d+1}]\}$ be a constructible special point data with $[u_1],\ldots, [u_{d+1}]$ written in lexicographical order. Let also $N_{\l}$ be a node of $C$ and $[v_1],\ldots, [v_{d+1}]$ be the $\ell$-ordering of $[u_1],\ldots, [u_{d+1}]$ with respect to the node $N_\l$. The following conditions holds.
\begin{enumerate}
\item The permutation $[v_1],\ldots, [v_{d+1}]$ of $[u_1],\ldots, [u_{d+1}]$ is cyclic.
\item $a_{[u_j],R}\leq a_{[u_{j+1}],R}$ for each $j=1,\ldots, d$.
\item $a^\l_{[v_j],R}\geq a^\l_{[v_{j+1}],R}$ for each $j=1,\ldots, d$.
\item $a^\l_{[v_1],R}-a^\l_{[v_{d+1}],R}\leq 1$; furthermore, the equality holds if and only if there exists $i\in\{1,\ldots,d\}$ such that $\l_i=\l$.
\item $|a_{[u_{j+1}],R}|-|a_{[u_j],R}|\leq 1$ for each $j=1,\ldots, d$.
\end{enumerate}
\end{Prop}

\begin{proof}
We proceed by induction on $d$. For $d=1$ the constructible special point data is of the form $\{(\l),[1],[2]\}$ and hence it satisfies all the stated conditions. Now, assume that these conditions  hold for $d$. Let $N_{\l_{d+1}}$ be a node of $C$, and let $R':=\{(\l_1,\ldots,\l_{d+1}),[w_1],\ldots, [w_{d+2}]\}$ be a special point of $\wt{\C}^{d+1}$ over $(R,N_{\l_{d+1}})$. 

 We begin by proving item (1). Let $[\wt{v}_1],\ldots, [\wt v_{d+1}]$ be the $\l_{d+1}$-ordering of $[u_1],\ldots,[u_{d+1}]$. It follows from Equation \eqref{eq:special} that $[w_1],\ldots, [w_{d+2}]$ is a permutation of 
\begin{equation}
\label{eq:v}
[\wt v_1\; 1],\ldots, [\wt v_h\;1],[\wt v_h\;2],\ldots,[\wt v_{d+1}\; 2].
\end{equation}
By Lemma~\ref{lem:ordering} the $\l_{d+1}$-ordering of the Equation \eqref{eq:v} is
\begin{equation}
\label{eq:wtv}
[\wt v_h\;2],\ldots,[\wt v_{d+1}\; 2], [\wt v_1\; 1],\ldots, [\wt v_h\;1].
\end{equation}
Let $[{v}_1],\ldots, [{v}_{d+1}]$ be the $\l$-ordering of $[u_1],\ldots, [u_{d+1}]$ for $\l\neq \l_{d+1}$. By induction hypothesis, the following relations hold for some $h_0$
\[
\wt v_{h_0}={v}_{1},\;\ldots,\; \wt v_{d+1}={v}_{d+2-h_0},\; \wt v_{1}={v}_{d+3-h_0},\;\ldots, \wt v_{h_0-1}={v}_{d+1}.
\]
If $h_0\leq h$, then it follows from Lemma~\ref{lem:ordering} that the $\l$-ordering of Equation \eqref{eq:v} is
\begin{equation}
\label{eq:wtvh1}
[\wt v_{h_0}\; 1],[\wt v_{h_0+1}\; 1],\ldots,[\wt v_h\; 1],[\wt v_h\; 2],\ldots, [\wt v_{d+1} \; 2], [\wt v_1\; 1],\ldots, [\wt v_{h_0-1}\; 1].
\end{equation}
If $h_0>h$, then the $\l$-ordering of Equation \eqref{eq:v} is
\begin{equation}
\label{eq:wtvh2}
[\wt v_{h_0}\; 2],[\wt v_{h_0+1}\; 2],\ldots,[\wt v_{d+1}\; 2],[\wt v_1\; 1],\ldots, [\wt v_{h} \; 1], [\wt v_h\; 2],\ldots, [\wt v_{h_0-1}\; 2].
\end{equation}

 Since by induction hypothesis, $[\wt v_1],\ldots,[\wt v_{d+1}]$ is a cyclic permutation of $[u_1],\ldots,[u_{d+1}]$ and the latter is in lexicographical order, the following relations hold for some  $h'$ 
\[
\wt v_{h'}=u_1,\;\ldots,\; \wt v_{d+1}=u_{d+2-h'}, \;\wt v_1=u_{d+3-h'},\; \ldots, \;\wt v_{h'-1}=u_{d+1}.
\]
If $h'\leq h$, then the lexicographical ordering of Equation \eqref{eq:v} is
\begin{equation}
\label{eq:vh1}
[\wt v_{h'}\; 1],\; [\wt v_{h'+1}\; 1],\;\ldots, \;[\wt v_h\; 1],\; [\wt v_h\; 2],\;\ldots, [\wt v_{d+1}\; 2],\; [\wt v_1\; 1],\;\ldots, [\wt v_{h'-1}\; 1].
\end{equation}
If $h'>h$, then the lexicographical ordering of Equation \eqref{eq:v} is
\begin{equation}
\label{eq:vh2}
[\wt v_{h'}\; 2],\; [\wt v_{h'+1}\; 2],\;\ldots, \;[\wt v_{d+1}\; 2],\; [\wt v_1\; 1],\;\ldots, [\wt v_h\; 1],\; [\wt v_h\; 2],\;\ldots, [\wt v_{h'-1}\; 2].
\end{equation}
We conclude that if $[w_1],\ldots, [w_{d+2}]$ is in lexicographical order (see Equations \eqref{eq:vh1} and \eqref{eq:vh2}), then their $\l$-ordering (see Equations \eqref{eq:wtvh1} and \eqref{eq:wtvh2}) is obtained by a cyclic permutation. The proof of the first item is complete.

 Recall that, by induction hypothesis, items (2), (3) and (4) hold for $d$. We want to prove that these items also hold for $d+1$. We can rewrite Equations \eqref{eq:vh1} and \eqref{eq:vh2} as follows
\begin{equation}
\label{eq:uh1}
\begin{array}{lr}
[u_1\; 1],\;\ldots, \;[u_{h+1-h'}\; 1],\; [u_{h+1-h'}\; 2],\;\ldots & \\
&\hspace{-1in} \ldots,[u_{d+2-h'}\; 2],\; [u_{d+3-h'}\; 1],\;\ldots, [u_{d+1}\; 1].
\end{array}
\end{equation}
and
\begin{equation}
\label{eq:uh2}
\begin{array}{lr}
[u_1\; 2],\;\ldots, \;[u_{d+2-h'}\; 2],\; [u_{d+3-h'}\; 1],\;\ldots & \\
&\hspace{-1.3in}\ldots, [u_{d+2+h-h'}\; 1],\; [u_{d+2+h-h'}\; 2],\;\ldots, [u_{d+1}\; 2].
\end{array}
\end{equation}

   If $\l\neq \l_{d+1}$, we have that $a^\l_{[u_j\; \epsilon],R'}=a^\l_{[u_j],R}$ for $\epsilon=1,2$; moreover, the $\l-$ordering of $[w_1],\ldots, [w_{d+2}]$ is essentially the same of $[u_1],\ldots, [u_{d+1}]$, as we can see in Equations \eqref{eq:wtvh1} and \eqref{eq:wtvh2}. Therefore, in this case, we have nothing to prove.\par
    Consider now the case $\l=\l_{d+1}$. We have 
\[
a^{\l_{d+1}}_{[u_j\; 1],R'}=a^{\l_{d+1}}_{[u_j],R}\quad\text{and}\quad a^{\l_{d+1}}_{[u_j\;2],R'}=a^{\l_{d+1}}_{[u_j],R}+1.
\]
We have now two cases. \par
\smallskip

\textbf{Case 1}. Assume that there exists $i\in\{1,\ldots,d\}$ such that $\l_i=\l_{d+1}$. Using the induction hypothesis, it is easy to see that the following relations hold 
\[
a^{\l_{d+1}}_{[u_1],R}+1=\ldots=a^{\l_{d+1}}_{[u_{d+2-h'}],R}+1=a^{\l_{d+1}}_{[u_{d+3-h'}],R}=\ldots=a^{\l_{d+1}}_{[u_{d+1}],R}.
\]
Note that the following relations also hold
\[
a^{\l_{d+1}}_{[u_{d+2-h'}\;2],R'}=a^{\l_{d+1}}_{[u_{d+2-h'}],R}+1=a^{\l_{d+1}}_{[u_{d+3-h'}],R}=a^{\l_{d+1}}_{[u_{d+3-h'}\;1],R'}.
\]
Therefore, with respect to Equation \eqref{eq:uh1}, we have
\[
a^{\l_{d+1}}_{[u_1\;1],R'}+1=\ldots=a^{\l_{d+1}}_{[u_{h+1-h'}\;1],R'}+1=a^{\l_{d+1}}_{[u_{h+1-h'}\;2],R'}=\ldots=a^{\l_{d+1}}_{[u_{d+1}\;2],R'},
\]   
while with respect to Equation \eqref{eq:uh2}, we have
\[
\begin{array}{lr}
a^{\l_{d+1}}_{[u_1\;2],R'}+2=\ldots=a^{\l_{d+1}}_{[u_{d+2+h-h'}\;1],R'}+2= & \\ &
 \hspace{-1in}=a^{\l_{d+1}}_{[u_{d+2+h-h'}\;2],R'}+1=\ldots=a^{\l_{d+1}}_{[u_{d+1}\;2],R'}+1.
 \end{array}
\]      
This proves item (2). 
To prove items (3) and (4), we just note that the $\l_{d+1}$-ordering of $[w_1],\ldots,[w_{d+2}]$ is given by Equation \eqref{eq:wtv}; moreover, in the case of Equation \eqref{eq:uh1} we have $\wt v_h=u_{h+1-h'}$, while in the case of Equation \eqref{eq:uh2} we have $\wt v_h=u_{d+2+h-h'}$. In particular, we have equality in item (4). \par
\smallskip

\textbf{Case 2}. Assume that there is no $i\in\{1,\ldots,d\}$ with $\l_i=\l_{d+1}$. Then the $\l_{d+1}$-ordering of $[u_1],\ldots,[u_{d+1}]$ is simply $[u_1],\ldots,[u_{d+1}]$.  In this case, using Equation \eqref{eq:v} (with $\wt v_i=u_i$) and the fact that $a^{\l_{d+1}}_{[u_j],R}=0$, we see that items (2), (3) and (4) readily hold.\par
  Finally, item (5) follows from Equations \eqref{eq:uh1} and \eqref{eq:uh2}, observing that
\[
|a_{[u_j\;\epsilon],R'}|=|a_{[u_j],R}|+\epsilon-1.
\]   
\end{proof}

\begin{proof}[Proof of Theorem \ref{thm:main}]
To conclude the proof of Theorem \ref{thm:main}, we have to check the conditions of Theorem~\ref{thm:map} for every special point $R$ of $\wt{\C}^d$. With the notation of this section, these conditions become
\begin{enumerate}
\item[(1)] for every $j_1,j_2=1,\ldots,d+1$ and every node $N_\l$ of $C$, we have
\[
|(a^\l_{[u_{j_1}],R}-b_{[u_{j_1}],R})-(a^\l_{[u_{j_2}],R}-b_{[u_{j_2}],R})|\leq 1.
\]
\item[(2)] for every $j_1,\ldots,j_q\in\{1,\ldots,d+1\}$, we have
\[
-\frac{q}{2}<\deg(\L|_{C_1})-e_{C_1}+\sum_{\l=1}^q(a^\l_{[u_{j_\l}],R}-b_{[u_{j_\l}],R})\leq \frac{q}{2}.
\]
\end{enumerate}

First, we note that by item (2) of Proposition~\ref{prop:main} and by Equation \eqref{eq:b} we have 
\[
b_{[u_i],R}\leq b_{[u_{i+1}],R}\quad\text{for every }i=1,\ldots, d.
\]
Moreover, by item (4) of Proposition~\ref{prop:main}, we get  $|a_{[u_{d+1}],R}|-|a_{[u_1],R}|\leq q$ and hence $b_{[u_{d+1}],R}-b_{[u_1],R}\leq 1$.

 We now prove condition (1). Assume without loss of generality that $j_1>j_2$. By items (2) and (4) of Proposition~\ref{prop:main} we see that 
\[
0\leq a^\l_{[u_{j_1}],R}-a^\l_{[u_{j_2}],R}\leq 1
\]
and, by the observation above, that
\[
0\leq b_{[u_{j_1}],R}-b_{[u_{j_2}],R}\leq 1.
\]
Therefore, condition (1) holds.

   As for condition (2), we just have to compute the minimum and maximum of the function
\[
F(j_1,\ldots,j_q)=\sum_{\l=1}^q(a^\l_{[u_{j_\l}],R}-b_{[u_{j_\l}],R}).
\]
Clearly it is enough to find the minimum and maximum of each function
\[
F_\l(j):=a^\l_{[u_j],R}-b_{[u_{j}],R}.
\]
Since $b_{[u_{i}],R}\leq b_{[u_{i+1}],R}$ and $b_{[u_{d+1}],R}-b_{[u_1],R}\leq 1$, we have two cases. In the first case, we have $b_{[u_i],R}=b_{[u_j],R}$ for every $i,j\in\{1,\ldots,d+1\}$; in the second case, there exists $h$ such that the following relations hold
\[
b_{[u_1],R}+1=\ldots=b_{[u_h],R}+1=b_{[u_{h+1}],R}=\ldots=b_{[u_{d+1}],R}.
\]
In the first case, it follows from item (2) of Proposition~\ref{prop:main} that the mimimum of $F_\l$ is attained at $j=1$, while the maximum is attained at $j=d+1$. On the other hand, in the second case, we claim that the minimum of $F_\l$ is attained at $j=h+1$. Indeed, using item (4) of Proposition~\ref{prop:main}, we see that for every $j\leq h$ we have
\[
a^\l_{[u_{h+1}],R}-b_{[u_{h+1}],R}=(a^\l_{[u_{h+1}],R}-1)-b_{[u_j],R}\leq a^\l_{[u_j],R}-b_{[u_j],R}.
\]
On the other hands, using item (2) of Proposition~\ref{prop:main}, we see that for every $j>h+1$ we have
\[
a^\l_{[u_{h+1}],R}-b_{[u_{h+1}],R}=a^\l_{[u_{h+1}],R}-b_{[u_j],R}\leq a^\l_{[u_j],R}-b_{[u_j],R}.
\]
 Similarly, one can show that the maximum of $F_\l$ is attained at $j=h$.\par
 By the arguments above, there exists some $h$, such that the minimum (respectively maximum) of $F(j_1,\ldots,j_q)$ is attained at $(h,h,\ldots,h)$. It follows from Proposition~\ref{prop:generic} that the sum 
\[
\deg(\L_{C_1})-e_{C_1}+\sum_{\l=1}^q (a^\l_{[u_h],R}-b_{[u_h],R})
\]
 satisfies condition (2). This concludes the proof of Theorem \ref{thm:main}.
\end{proof}


\bigskip
\noindent{\smallsc Alex Abreu, Universidade Federal Fluminense,\\ Rua M. S. Braga, s/n, Valonguinho, 24020-005 Niter\'oi (RJ) Brazil.}\\
{\smallsl E-mail address: \small\verb?alexbra1@gmail.com?}

\bigskip
\bigskip
\noindent{\smallsc Juliana Coelho, Universidade Federal Fluminense,\\ Rua M. S. Braga, s/n, Valonguinho, 24020-005 Niter\'oi (RJ) Brazil}\\
{\smallsl E-mail address: \small\verb?julianacoelho@vm.uff.br?}
\bigskip
\bigskip

\noindent{\smallsc Marco Pacini, Universidade Federal Fluminense,\\ Rua M. S. Braga, s/n, Valonguinho, 24020-005 Niter\'oi (RJ) Brazil}\\
{\smallsl E-mail address: \small\verb?pacini@impa.br? and \small\verb?pacini@vm.uff.br?}

\end{document}